\documentclass[10pt]{article}
\usepackage{amssymb}
\usepackage{amsmath}
\usepackage{amsthm}
\usepackage{amsfonts}

\usepackage{a4wide}
\usepackage{longtable}
\usepackage{multirow}

\usepackage[latin1]{inputenc}
\usepackage[english]{babel}
\usepackage[T1]{fontenc}
\usepackage{amssymb,amsmath,amsthm,amstext,amsfonts,latexsym}

\usepackage{pgf,tikz}
\usepackage{mathrsfs}
\usetikzlibrary{arrows}

\newtheorem{lemma}{Lemma}[section]

\newtheorem{theorem}{Theorem}[section]
\newtheorem{corollary}{Corollary}[theorem]

\theoremstyle{definition}
\newtheorem{definition}{Definition}[section]

\theoremstyle{remark}
\newtheorem{remark}{Remark}[section]
\newtheorem{example}{Example}[section]
\date{ }
\usepackage[all]{xy}
\input xy
\xyoption {all}
\usepackage[pdftex,           %
    bookmarks=true,           
    hypertexnames=false,      
    bookmarksnumbered=true,   
    pdfpagemode=None,         
    pdfstartview=FitH,        
    pdfpagelayout=SinglePage, 
    colorlinks=true,          
    urlcolor=blue,            
    citecolor=blue,           
    pdfborder={0 0 0}         
    ]{hyperref}
    
\begin{document}

\title{$3$-rank of ambiguous class groups \\ of cubic Kummer extensions}

\author{S. AOUISSI, D. C. MAYER, M. C. ISMAILI, M. TALBI and A. AZIZI}


\maketitle

\noindent
\begin{center}
Dedicated to the memory of Frank Emmett Gerth III
\end{center}
\bigskip

\noindent
\textbf{Abstract:}
Let $k=k_0(\sqrt[3]{d})$ be a cubic Kummer extension of $k_0=\mathbb{Q}(\zeta_3)$
with $d>1$ a cube-free integer
and $\zeta_3$ a primitive third root of unity.
Denote by 
$C_{k,3}^{(\sigma)}$
the $3$-group of ambiguous classes of the extension $k/k_0$
with relative group $G=\operatorname{Gal}(k/k_0)=\langle\sigma\rangle$.
The aims of this paper are
 to characterize all extensions $k/k_0$ with cyclic $3$-group of ambiguous classes $C_{k,3}^{(\sigma)}$ of order $3$,
 to investigate the multiplicity $m(f)$ of the conductors $f$ of these abelian extensions $k/k_0$,
 and to classify the fields $k$ according to the cohomology of their unit groups $E_{k}$ as Galois modules over $G$.
The techniques employed for reaching these goals are
relative $3$-genus fields, Hilbert norm residue symbols, quadratic $3$-ring class groups modulo $f$,
the Herbrand quotient of $E_{k}$, and central orthogonal idempotents.
All theoretical achievements are underpinned by extensive computational results.

\bigskip
\noindent
{\bf Keywords:} {Pure cubic fields, cubic Kummer extensions, $3$-group of ambiguous ideal classes, $3$-rank,
Hilbert $3$-class field, relative $3$-genus field, multiplicity of conductors,
Galois cohomology of unit groups, principal factorization types.}

\bigskip
\noindent
{\bf Mathematics Subject Classification 2010:} {11R11, 11R16, 11R20, 11R27, 11R29, 11R37.}
\bigskip


\section{Introduction}
\label{s:Intro}

Let $d>1$ be a cube-free integer and
$k=\mathbb{Q}(\sqrt[3]{d},\zeta_3)$ 
be a cubic Kummer extension of the cyclotomic field $k_0=\mathbb{Q}(\zeta_3)$.
Denote by $f$ the conductor of the abelian extension $k/k_0$, by $m=m(f)$ its multiplicity, and by
$C_{k,3}^{(\sigma)}$ the $3$-group of ambiguous ideal classes of $k/k_0$.
Let $k^{\ast}=(k/k_0)^{\ast}$ be the maximal abelian extension of $k_0$
contained in the Hilbert $3$-class field $k_1$ of $k$,
which is called the relative $3$-genus field of $k/k_0$ (cf. \cite[\S\ 2, p. VII-3]{Hz}).

We consider the problem of finding the radicands \(d\) and conductors \(f\)
of all pure cubic fields $L=\mathbb{Q}(\sqrt[3]{d})$
for which the Galois group $\operatorname{Gal}(k^{\ast}/k)$ is non-trivial cyclic.
The present work gives the complete solution of this problem
by characterizing all cubic Kummer extensions $k/k_0$ with cyclic $3$-group of ambiguous ideal classes $C_{k,3}^{(\sigma)}$ of order $3$.
In fact, we prove the following Main Theorem:


\begin{theorem}
\label{thm:Rank1}
Let $k=\mathbb{Q}(\sqrt[3]{d},\zeta_3)$, where $d>1$ is a cube-free integer,
and $C_{k,3}^{(\sigma)}$ be the $3$-group of ambiguous ideal classes of $k/\mathbb{Q}(\zeta_3)$.
Then, $\operatorname{rank}\,(C_{k,3}^{(\sigma)})=1$ if and only if the integer $d$ can be written in one of the following forms:
\begin{equation}
\label{eqn:Rank1}
 d=\left\lbrace
   \begin{array}{ll}
   p_1^{e_1} & \text{ with } p_1\equiv 1\,(\mathrm{mod}\,3),\\
   3^{e}p_1^{e_1} & \text{ with } p_1\equiv 4 \text{ or } 7\,(\mathrm{mod}\,9),\\
   p_1^{e_1}q_1^{f_1}\equiv\pm 1\,(\mathrm{mod}\,9) & \text{ with } p_1,-q_1\equiv 4 \text{ or } 7\,(\mathrm{mod}\,9),\\
   3^{e}q_1^{f_1} & \text{ with } q_1\equiv -1\,(\mathrm{mod}\,9),\\
   q_1^{f_1}q_2^{f_2} & \text{ with } q_1\equiv q_2\equiv -1\,(\mathrm{mod}\,9),\\
   q_1^{f_1}q_2^{f_2}\not\equiv\pm 1\,(\mathrm{mod}\,9) & \text{ with } q_{i}\not\equiv -1\,(\mathrm{mod}\,9) \text{ for some } i\in\lbrace 1,2\rbrace,\\
   3^{e}q_1^{f_1}q_2^{f_2} & \text{ with } q_{i}\not\equiv -1\,(\mathrm{mod}\,9) \text{ for some } i\in\lbrace 1,2\rbrace,\\
   q_1^{f_1}q_2^{f_2}q_3^{f_3}\equiv\pm 1\,(\mathrm{mod}\,9) & \text{ with } q_1,q_2\equiv 2 \text{ or } 5\,(\mathrm{mod}\,9) \text{ and } q_3\equiv -1\,(\mathrm{mod}\,9),\\
   q_1^{f_1}q_2^{f_2}q_3^{f_3}\equiv\pm 1\,(\mathrm{mod}\,9) & \text{ with } q_1,q_2,q_3\equiv 2 \text{ or } 5\,(\mathrm{mod}\,9),\\
   \end{array}
   \right.
\end{equation}
where $p_1\equiv 1\,(\mathrm{mod}\,3)$ and $q_1,q_2,q_3\equiv -1\,(\mathrm{mod}\,3)$ are primes and
$e, e_1, f_1, f_2$ and $f_3$ are integers equal to $1$ or $2$.
\end{theorem}

We should point out that Theorem \ref{thm:Rank1} is expressed in terms of radicands $d>1$. However,
in Theorem \ref{thm:Ismaili1} of section \ref{ss:SplitPrime} and Theorem \ref{thm:Ismaili2} of section \ref{ss:NonSplitPrimes}
we emphasize that an ambiguous class group $C_{k,3}^{(\sigma)}$ of $3$-rank $1$ is rather a characteristic property of
the class field theoretic conductor \(f=3^e\cdot\ell_1\cdots\ell_n\) of the cyclic extension $k/k_0$,
i.e. an essentially square-free positive integer,
divisible by the prime factors \(\ell_j\) of the radicand \(d=3^{e_0}\cdot\ell_1^{e_1}\cdots\ell_n^{e_n}\)
(and possibly additionally by the prime $3$)
but independent of the exponents \(e_1,\ldots,e_n\).
Theorems \ref{thm:Honda}, \ref{thm:Ismaili1} and \ref{thm:Ismaili2} show that
pure cubic fields $L$ can be collected in multiplets \((L_1,\ldots,L_m)\)
sharing a common conductor $f$ with multiplicity $m$
and a common type of ambiguous class group $C_{k,3}^{(\sigma)}$ of their associated Galois closures $k$.

One of the motivations of this work is an earlier article of Honda \cite{Ho}.
Honda's result concerning trivial ambiguous class groups $C_{k,3}^{(\sigma)}$ of $3$-rank $0$
is restated in terms of conductors $f$ and multiplicities $m$
in section \ref{ss:Rank0}, Theorem \ref{thm:Honda},
to enable immediate comparison with the results of our paper,
which continues Honda's explicit description of radicands $d$
from trivial to cyclic ambiguous $3$-class groups $C_{k,3}^{(\sigma)}$ of $3$-rank $1$.
Another motivation is that our results on the factorization of the conductor $f$ of $k$ into prime numbers   
provide information about the ramification in $k$,
allow the study of the class group of the relative genus field $k^{\ast}$ of $k/k_0$,
and can be applied to analyze the class field tower of $k$.

A further main result of our paper is Theorem \ref{thm:PrincipalFactorization},
where we use the Galois cohomology and the Herbrand quotient
of the unit group \(E_{k}\) of the normal closure $k$
as a module over the group \(G=\mathrm{Gal}(k/k_0)=\langle\sigma\rangle\)
to give a brief and elegant new proof for
the existence of exactly three principal factorization types
\(\alpha,\beta,\gamma\) of pure cubic fields,
unifying concepts introduced by Barrucand and Cohn \cite[Thm. 15.6]{BC1971} and corrected by Halter-Koch \cite[Korollar, p. 594]{HK}.

In the numerical Examples \ref{exm:PrincipalFactorization}, \ref{exm:Honda}, \ref{exm:Ismaili1} and \ref{exm:Ismaili2},
  we present statistical results of
the most extensive computation of invariants of pure cubic fields up to now.
They cover all $827600$ non-isomorphic fields $L$ with normalized radicands $1<d<10^6$,
determined with Voronoi's algorithm \cite{Vo} and confirmed with Magma \cite{MAGMA},
significantly extending results of H. C. Williams in
\cite{HCW1982}.

The outline of this work is as follows.
After having illuminated the broader scope of our main results in section \ref{s:Context},
we give the proof of Theorem \ref{thm:Rank1} in section \ref{s:Proof}.
The section  \ref{s:FullClassGroup} summarizes the structures of
the $3$-class group of the pure cubic field $L$
and the full $3$-class group of the normal closure $k$.
In section \ref{s:Conclusion}, 
we conclude the paper with an important \textit{open problem} to a broader audience.
The Outlook of section \ref{s:Outlook} investigates another related scenario
for the location of the relative $3$-genus field $k^{\ast}$.
 \\


$\textbf{Notations. \ }$
Throughout this paper, we use the following notation:
\begin{itemize}
 \item $p$ and $q$ are prime numbers such that $p\equiv 1\,(\mathrm{mod}\,3)$ and $q\equiv -1\,(\mathrm{mod}\,3)$;
 \item $L=\mathbb{Q}(\sqrt[3]{d})$: a pure cubic field, where $d\neq 1$ is a cube-free positive integer;
 \item $k_0=\mathbb{Q}(\zeta_3)$, with $\zeta_3=e^{2i\pi/3}=\frac{1}{2}(-1+\sqrt{-3})$ a primitive cuberoot of unity;
 \item $k=\mathbb{Q}(\sqrt[3]{d},\zeta_3)$: the normal closure of $L$, a cubic Kummer extension of $k_0$;
 \item $f$: the conductor of the abelian extension $k/k_0$; 
 \item $m=m(f)$: the multiplicity of the conductor $f$;
 \item $\langle\tau\rangle=\operatorname{Gal}(k/L)$ such that $\tau^2=id,\ \tau(\zeta_3)=\zeta_3^2$ and $\tau(\sqrt[3]{d})=\sqrt[3]{d}$;
 \item $\langle\sigma\rangle=\operatorname{Gal}(k/k_0)$ such that $\sigma^3=id,\ \sigma(\zeta_3)=\zeta_3$ and $\sigma(\sqrt[3]{d})=\zeta_3\sqrt[3]{d}$;
 \item $\lambda=1-\zeta_3$ and $\pi$ are prime elements of $k_0$;
 \item $q^{\ast}=0$ or $1$ according to whether $\zeta_{3}$ is not norm or is norm of an element of $k^\times=k\setminus\lbrace 0\rbrace$;
 \item $t$: the number of prime ideals of $k_{0}$ ramified in $k$, $2s$ among them split over $\mathbb{Q}$.
 \item For a number field $F$, denote by $F_1$ the Hilbert $3$-class field of $F$ and by

  \begin{itemize}
   \item $\mathcal{O}_{F}$, $\mathcal{I}_{F}$: the ring of integers, and the group of ideals of $F$;
   \item $E_{F}$, $\mathcal{P}_{F}$: the group of units, and the group of principal ideals of $F$;
   \item $C_{F}$, $h_{F}$, $C_{F,3}$: the class group, class number, and $3$-class group of $F$.
  \end{itemize}
  
\end{itemize}



\section{Viewing the main result in a wider perspective}
\label{s:Context}

In this section, our intention is to shed light on
the structure of the entire $3$-class group $C_{k,3}$ of the normal closure $k$ of $L$
which is minorized by the $3$-group $C_{k,3}^{(\sigma)}$ of ambiguous ideal classes of $k/k_0$.
The investigation of this relationship requires the distinction of
\textit{principal factorization types} in the sense of Barrucand and Cohn \cite{BC1971}.
We present a brief and elegant new proof of this classification by means of
the Galois cohomology of the unit group $E_{k}$ as a module over $\operatorname{Gal}(k/k_0)=\langle\sigma\rangle$,
since the original paper \cite{BC1971} contained a superfluous type,
which was shown to be impossible by Halter-Koch \cite{HK},
and therefore no complete proof is available in a single article of the existing literature.



\subsection{Galois cohomology of the unit group and principal factorization types}
\label{ss:Cohomology}

Let $E_{k}$ be the unit group of $k=\mathbb{Q}(\sqrt[3]{d},\zeta_3)$
and denote by $Q=(E_{k}:E_{0})$ the index of the subgroup $E_{0}$
which is generated by all units of proper subfields of $k$, that is,
$E_{0}=\langle E_{k_0},E_{L},E_{L^{\sigma}},E_{L^{\sigma^2}}\rangle$.
Pierre Barrucand and Harvey Cohn have based their
classification of pure cubic fields $L=\mathbb{Q}(\sqrt[3]{d})$
into \textit{principal factorization types} \cite[Thm. 15.6, pp. 235--236]{BC1971}
on the \textit{class number relation} $h_{k}=\frac{Q}{3}\cdot h_{L}^2$
\cite[Thm. 14.1, p. 232]{BC1971}, \cite[Lem. 1, p. 7]{Ho},
which involves the \textit{subfield unit index} $Q\in\lbrace 1,3\rbrace$
\cite[Thm. 12.1, p. 229]{BC1971}.

Since a superfluous type of Barrucand and Cohn was eliminated by Franz Halter-Koch \cite[Korollar, p. 594]{HK},
we present an alternative access to this classification
by means of the Galois cohomology of the unit group $E_{k}$ of the normal closure $k$
with respect to the relative automorphism group $G=\operatorname{Gal}(k/k_0)=\langle\sigma\rangle$.
Instead of \(Q\), the primary invariant is the cohomology group
$H^0(G,E_{k})=E_{k_0}/N_{k/k_0}(E_{k})$
and its order, the \textit{unit norm index} $3^U=(E_{k_0}:N_{k/k_0}(E_{k}))$
where \(U\) can only take two values $U\in\lbrace 0,1\rbrace$,
according to whether $\zeta_3$ is the relative norm of a unit of $k$ or not.
(Observe that $E_{k_0}=\langle -1,\zeta_3\rangle$ and $-1=(-1)^3=N_{k/k_0}(-1)$ is a norm.)
The Theorem on the Herbrand quotient of $E_{k}$ \cite[Thm. 3, p. 92]{Hb},
which is a quantitative version of Hilbert's Theorem $92$ \cite[\S\ 56, p. 275]{Hi},
admits the calculation of the order of the cohomology group
$H^1(G,E_{k})\simeq (E_{k}\cap\ker(N_{k/k_0}))/E_{k}^{1-\sigma}$
by the general formula
\begin{equation}
\label{eqn:Herbrand}
\#H^1(G,E_{k})=\#H^0(G,E_{k})\cdot\lbrack k:k_0\rbrack,
\end{equation}
since no real archimedean place of $k_0$ becomes complex in $k$.

According to Iwasawa \cite[\S\ 2, pp. 189--190]{Iw},
the quotient group $(E_{k}\cap\ker(N_{k/k_0}))/E_{k}^{1-\sigma}$
is isomorphic to the group of \textit{primitive ambiguous principal ideals}
$\mathcal{P}_{k}^G/\mathcal{P}_{k_0}$
of $k$ with respect to $k_0$.
Equation \eqref{eqn:Herbrand} implies that the order $\#H^1(G,E_{k})=3^U\cdot 3\in\lbrace 3,9\rbrace$
can take precisely two values in dependence on the unit norm index $3^U$,
whereas Hilbert's Theorem $92$ only states that $H^1(G,E_{k})$ is non-trivial.
Since $k_0$ is a principal ideal domain, we have $\mathcal{I}_{k_0}=\mathcal{P}_{k_0}$
and $\mathcal{P}_{k}^G/\mathcal{P}_{k_0}$ is a subgroup of the group
$\mathcal{I}_{k}^G/\mathcal{I}_{k_0}$
of \textit{primitive ambiguous ideals} of $k/k_0$.
By determining the order of the group $\mathcal{I}_{k}^G/\mathcal{I}_{k_0}$
we occasionally can estimate the order of the subgroup $\mathcal{P}_{k}^G/\mathcal{P}_{k_0}$
by a better upper bound than \(3^{U+1}\).
As a refinement of the classification by the Galois cohomology, however,
we first intend to split the group $\mathcal{I}_{k}^G/\mathcal{I}_{k_0}$
into the direct product of the absolute component $\mathcal{I}_{L}^G/\mathcal{I}_{\mathbb{Q}}$
and a relative complement.
The splitting can be accomplished by means of image and kernel
$\mathcal{I}_{k}^G/\mathcal{I}_{k_0}\simeq\mathrm{img}(N_{k/L})\times\ker(N_{k/L})$
of the relative norm $N_{k/L}$
or, equivalently, by the action
$\mathcal{I}_{k}^G/\mathcal{I}_{k_0}\simeq(\mathcal{I}_{k}^G/\mathcal{I}_{k_0})^{1+\tau}\times(\mathcal{I}_{k}^G/\mathcal{I}_{k_0})^{1-\tau}\)
of the central orthogonal idempotents $\frac{1}{2}(1+\tau)$ and $\frac{1}{2}(1-\tau)$
in the group ring $\mathbb{F}_3\lbrack\tau\rbrack$
of the relative group $\operatorname{Gal}(k/L)=\langle\tau\rangle$
over the finite field $\mathbb{F}_3$.
We denote by $t$ the number of all prime ideals of $k_{0}$ which ramify in $k$,
and by $2s$ the number of those among them lying over primes which split in $k_{0}/\mathbb{Q}$.
Then Hilbert's Theorem $93$ \cite[\S\ 57, p. 277]{Hi} yields the
orders of elementary abelian $3$-groups, respectively dimensions of vector spaces over $\mathbb{F}_3$, involved:
\begin{equation}
\label{eqn:Hilbert93}
\mathcal{I}_{k}^G/\mathcal{I}_{k_0}\simeq\mathbb{F}_3^t\simeq
\mathcal{I}_{L}^G/\mathcal{I}_{\mathbb{Q}}\times\ker(N_{k/L})\simeq\mathbb{F}_3^{t-s}\times\mathbb{F}_3^s,
\end{equation}
since the norm map $N_{k/L}$ induces an epimorphism
from $\mathcal{I}_{k}^G/\mathcal{I}_{k_0}$ onto $\mathcal{I}_{L}^G/\mathcal{I}_{\mathbb{Q}}$.
Finally, the splitting of $\mathcal{I}_{k}^G/\mathcal{I}_{k_0}$
restricts to the subgroup $\mathcal{P}_{k}^G/\mathcal{P}_{k_0}$ and we obtain the relation
\begin{equation}
\label{eqn:Principal}
\#(\mathcal{P}_{k}^G/\mathcal{P}_{k_0})=3^{U+1}=
\#(\mathcal{P}_{L}^G/\mathcal{P}_{\mathbb{Q}})\cdot\#(\ker(N_{k/L})\cap(\mathcal{P}_{k}^G/\mathcal{P}_{k_0}))=3^{A+R},
\end{equation}
where
$A=\dim_{\mathbb{F}_3}(\mathcal{P}_{L}^G/\mathcal{P}_{\mathbb{Q}})$
denotes the dimension of the subspace of \textit{absolute} principal factors, 
and
$R=\dim_{\mathbb{F}_3}(\ker(N_{k/L})\cap(\mathcal{P}_{k}^G/\mathcal{P}_{k_0}))$
denotes the dimension of the subspace of \textit{relative} principal factors.


With the preceding developments we have given a concise proof of the following Theorem.

\begin{theorem}
\label{thm:PrincipalFactorization}
Each pure cubic field $L=\mathbb{Q}(\sqrt[3]{d})$
belongs to precisely one of the following three
principal factorization types in dependence on the unit norm index
$3^U=(E_{k_0}:N_{k/k_0}(E_{k}))$
and on the pair of invariants $(A,R)$, where
$3^A=\#(\mathcal{P}_{L}^G/\mathcal{P}_{\mathbb{Q}})$
and
$3^R=\#(\ker(N_{k/L})\cap(\mathcal{P}_{k}^G/\mathcal{P}_{k_0}))$:


\renewcommand{\arraystretch}{1.5}

\begin{table}[ht]
\label{tbl:PrincipalFactorizationTypes}
\begin{center}
\begin{tabular}{|c|c|cc|}
\hline
 Type          & $U$ & $A$ & $R$ \\
\hline
 type $\alpha$ & $1$ & $1$ & $1$ \\
 type $\beta$  & $1$ & $2$ & $0$ \\
 type $\gamma$ & $0$ & $1$ & $0$ \\
\hline
\end{tabular}
\end{center}
\end{table}


\noindent
The invariants satisfy the following relations:
the equation $A+R=U+1$ and the estimates $1\le A+R\le t$, $1\le A\le t-s$, and $0\le R\le s$,
since for each type, the group $\mathcal{P}_{L}^G/\mathcal{P}_{\mathbb{Q}}$ of absolute principal factors contains
the subgroup $\Delta=\langle\sqrt[3]{d}\mathcal{O}_L\rangle$ of principal ideals generated by radicals.
\end{theorem}


\begin{remark}
\label{rmk:PrincipalFactorization}
Our deduction of \textit{three} possible principal factorization types $\alpha$, $\beta$, and $\gamma$
has the advantage of being very brief and elegant.
In the given order, they correspond to type III, I, and IV of Barrucand and Cohn \cite[Thm. 15.6, pp. 235--236]{BC1971}.
Type II has been proven to be impossible by Halter-Koch \cite[Korollar, p. 594]{HK}.
A drawback of our access is the lack of connections to the
class number relation $h_{k}=\frac{Q}{3}\cdot h_{L}^2$ and the subfield unit index $Q$,
which is given in the following manner by combining \cite{BC1971} and \cite{HK}:
\begin{equation}
\label{eqn:UnitIndex}
Q=1 \text{ for type } \alpha \quad \text{ and } \quad Q=3 \text{ for both types } \beta,\ \gamma.
\end{equation}
\end{remark}


\begin{example}
\label{exm:PrincipalFactorization}
By means of our own implementation of Voronoi's algorithm \cite{Vo},
we have determined the statistical distribution of the principal factorization types
over all pure cubic fields $L=\mathbb{Q}(\sqrt[3]{d})$ with normalized radicands $d<10^6$.
The total number of these fields, which are pairwise non-isomorphic by the normalization,
is $827\,600$.
The dominating part of $635\,463$ fields ($76.78\%$) is of type $\beta$,
$163\,527$ fields ($19.76\%$) are of type $\alpha$,
and only $28\,610$ fields ($3.46\%$) are of type $\gamma$.
($254\,254$, $382\,231$, $191\,115$ fields are of species $1\mathrm{a}$, $1\mathrm{b}$, $2$, respectively.)
Our results were confirmed with Magma \cite{MAGMA}.
They significantly extend the computations of H. C. Williams \cite[pp. 272--273]{HCW1982}.
\end{example}


Before we come to the application of the principal factorization types
to the relationship between the full $3$-class group $C_{k,3}$
and its subgroup $C_{k,3}^{(\sigma)}$ of ambiguous classes in \S\S\ \ref{ss:Rank0} -- \ref{ss:NonSplitPrimes},
we show in \S\ \ref{ss:Multiplicity} how to replace the radicands $d$ of pure cubic fields $L=\mathbb{Q}(\sqrt[3]{d})$
by class field theoretic conductors $f$ of the corresponding cyclic relative extensions $k/k_0$,
which enables the determination of the multiplicity $m(f)$ of non-isomorphic fields sharing a common conductor.



\subsection{Conductors and their multiplicity}
\label{ss:Multiplicity}

The class field theoretic \textit{conductor} $f$ of the Kummer extension $k/k_0$
is the smallest positive integer such that
$k$ is contained in the $3$-ring class field modulo $f$
of the quadratic field $k_0=\mathbb{Q}(\zeta_3)=\mathbb{Q}(\sqrt{-3})$.
If $L=\mathbb{Q}(\sqrt[3]{d})$ is a pure cubic field
with normalized cube-free radicand $d=d_1d_2^2$ constituted by square-free coprime integers
$d_1>d_2\ge 1$, $\gcd(d_1,d_2)=1$,
then $d$ is strictly smaller than the co-radicand, $d_1d_2^2<d_1^2d_2$,
and the conductor $f$ of the Galois closure $k=\mathbb{Q}(\sqrt[3]{d},\zeta_3)$
is given by the following formula (which also holds without normalization):
\begin{equation}
\label{eqn:ConductorHomogeneousComponents}
f=
\begin{cases}
d_1d_2 & \text{ if } d\equiv\pm 1\,(\mathrm{mod}\,9) \text{ (Dedekind's species 2)}, \\
3d_1d_2 & \text{ if } d\not\equiv\pm 1\,(\mathrm{mod}\,9) \text{ (Dedekind's species 1)}. \\
\end{cases}
\end{equation}
We see that $f$ is essentially square-free with the possible exception of its $3$-part.
However, equation
\eqref{eqn:ConductorHomogeneousComponents}
is too coarse for determining the \textit{multiplicity} $m(f)$ of the conductor $f$,
that is the number of non-isomorphic pure cubic fields sharing the common conductor $f$.
For this purpose, we need the prime factorization of the radicand $d$.
If $d=3^{e_0}\cdot\ell_1^{e_1}\cdots\ell_n^{e_n}$ with $n\ge 0$,
pairwise distinct prime numbers $\ell_j\neq 3$,
and exponents $0\le e_0\le 2$, $1\le e_j\le 2$ for $1\le j\le n$,
then the prime factorization of the conductor $f$ is given by
\begin{equation}
\label{eqn:ConductorPrimeFactorization}
f=3^e\cdot\ell_1\cdots\ell_n, \text{ where }
e=
\begin{cases}
0 & \text{ if } e_0=0,\ d\equiv\pm 1\,(\mathrm{mod}\,9) \text{ (species 2)}, \\
1 & \text{ if } e_0=0,\ d\not\equiv\pm 1\,(\mathrm{mod}\,9) \text{ (species 1b)}, \\
2 & \text{ if } e_0\ge 1 \text{ (species 1a)}.
\end{cases}
\end{equation}


\noindent
With equation \eqref{eqn:ConductorPrimeFactorization} we are in the position to express the multiplicity
by the formula in \cite[Thm. 2.1, p. 833]{Ma1992}.

\begin{theorem}
\label{thm:Multiplicity}
The multiplicity $m(f)$ of the conductor $f=3^e\cdot\ell_1\cdots\ell_n$ of the Kummer extension $k/k_0$
is given in dependence on the numbers

$u:=\#\lbrace 1\le j\le n\mid\ell_j\equiv\pm 1\,(\mathrm{mod}\,9)\rbrace$ and
$v:=\#\lbrace 1\le j\le n\mid\ell_j\equiv\pm 2,\pm 4\,(\mathrm{mod}\,9)\rbrace$,

\noindent
of prime divisors, respectively, on the total number $n:=u+v$ of all prime divisors of $f$ distinct from $3$,
by the formulas
\begin{equation}
\label{eqn:Multiplicity}
m(f)=
\begin{cases}
2^n & \text{ if } e=2 \text{ (species $\mathrm{1a}$)}, \\
2^u\cdot X_v & \text{ if } e=1 \text{ (species $\mathrm{1b}$)}, \\
2^u\cdot X_{v-1} & \text{ if } e=0 \text{ (species $\mathrm{2}$)},
\end{cases}
\end{equation}
where the sequence $(X_k)_{k\ge -1}$ is defined by $X_k=\frac{1}{3}(2^k-(-1)^k)$.
\end{theorem}

\noindent
From the broader perspective of arbitrary non-Galois cubic fields,
equation \eqref{eqn:Multiplicity} can also be derived from \cite{Ma2014}, namely
from Thm. 3.4, eqn. (3.4), p. 2217, for species 1a,
and, taking notice of Cor. 3.2, p. 2219,
from Thm. 3.3, eqn. (3.3), p. 2217, for species 1b and 2.


\begin{definition}
\label{dfn:Multiplet}
When $f$ is a conductor with multiplicity $m:=m(f)$
we say that the corresponding pairwise non-isomorphic pure cubic fields $L$ which share the common conductor $f$
form a \textit{multiplet} $(L_1,\ldots,L_m)$.
Their normalized \textit{companion radicands} $d_1,\ldots,d_m$
such that $L_i=\mathbb{Q}(\sqrt[3]{d_i})$
can be constructed from $f$
by varying the exponents $e_j$ of the prime factors $\ell_j$.
\end{definition}


\subsection{Determination of 3-ranks of class groups}
\label{ss:RankFormulas}

\noindent
An estimate for the $3$-class rank $r:=\operatorname{rank}\,(C_{L,3})$
of a pure cubic field $L=\mathbb{Q}(\sqrt[3]{d})$
has been given by Barrucand, Williams and Baniuk \cite[(2.1), p. 313]{BWB}
in the following form.

Let $\tilde{t}$ be the number of all prime divisors of the conductor $f$ of $k/k_0$,
$\tilde{s}$ be the number of those which are congruent to $1$ modulo $3$,
and $\tilde{v}$ be the number of those which are congruent to either $\pm 2$ or $\pm 4$ modulo $9$.
Put $\tilde{\varepsilon}:=0$ if $\tilde{v}=0$, and $\tilde{\varepsilon}:=1$ if $\tilde{v}\ge 1$,
that is $\tilde{\varepsilon}=\min(1,\tilde{v})$.
Then we have a lower bound and an upper bound for $r$ in terms of $\tilde{s}$ and $\tilde{\delta}:=\tilde{t}-1-\tilde{\varepsilon}$:
\begin{equation}
\label{eqn:BWB}
\max(\tilde{s},\tilde{\delta})\quad\le\quad r\quad\le\quad\tilde{s}+\tilde{\delta}.
\end{equation}


In all of our applications, the lower bound (maximum) and the upper bound (sum) will coincide,
and we shall obtain the precise $3$-class rank of $L$.

In Table \ref{tbl:BWB00}, we start with the fewest possible prime factors of the conductor $f$,
and none of them split in $k_0$, that is $\tilde{s}=0$.


\renewcommand{\arraystretch}{1.1}

\begin{table}[ht]
\caption{Pure cubic fields $L$ with $\tilde{s}=0$ and $r=0$}
\label{tbl:BWB00}
\begin{center}
\begin{tabular}{|cl|ccccc|cc|c|}
\hline
 Item  & $f$ & $\tilde{t}$ & $\tilde{s}$ & $\tilde{v}$ & $\tilde{\varepsilon}$ & $\tilde{\delta}=\tilde{t}-1-\tilde{\varepsilon}$ & $\max(\tilde{s},\tilde{\delta})$ & $\tilde{s}+\tilde{\delta}$ & $r$ \\
\hline
 $(1)$ & $9$                            & $1$ & $0$ & $0$ & $0$ & $0$ & $0$ & $0$ & $0$ \\
 $(2)$ & $q\equiv 8\,(9)$               & $1$ & $0$ & $0$ & $0$ & $0$ & $0$ & $0$ & $0$ \\
 $(3)$ & $3q$, $q\equiv 2,5\,(9)$       & $2$ & $0$ & $1$ & $1$ & $0$ & $0$ & $0$ & $0$ \\
 $(4)$ & $9q$, $q\equiv 2,5\,(9)$       & $2$ & $0$ & $1$ & $1$ & $0$ & $0$ & $0$ & $0$ \\
 $(5)$ & $q_1q_2$, $q_j\equiv 2,5\,(9)$ & $2$ & $0$ & $2$ & $1$ & $0$ & $0$ & $0$ & $0$ \\
\hline
\end{tabular}
\end{center}
\end{table}


In Table \ref{tbl:BWB11}, we proceed with conductors $f$ having the smallest numbers of prime divisors
such that exactly one of them splits in $k_0$, that is $\tilde{s}=1$.


\renewcommand{\arraystretch}{1.1}

\begin{table}[ht]
\caption{Pure cubic fields $L$ with $\tilde{s}=1$ and $r=1$}
\label{tbl:BWB11}
\begin{center}
\begin{tabular}{|cl|ccccc|cc|c|}
\hline
 Item  & $f$  & $\tilde{t}$ & $\tilde{s}$ & $\tilde{v}$ & $\tilde{\varepsilon}$ & $\tilde{\delta}=\tilde{t}-1-\tilde{\varepsilon}$ & $\max(\tilde{s},\tilde{\delta})$ & $\tilde{s}+\tilde{\delta}$ & $r$ \\
\hline
 $(1)$ & $p\equiv 1\,(9)$                             & $1$ & $1$ & $0$ & $0$ & $0$ & $1$ & $1$ & $1$ \\
 $(2)$ & $3p$, $p\equiv 4,7\,(9)$                     & $2$ & $1$ & $1$ & $1$ & $0$ & $1$ & $1$ & $1$ \\
 $(3)$ & $9p$, $p\equiv 4,7\,(9)$                     & $2$ & $1$ & $1$ & $1$ & $0$ & $1$ & $1$ & $1$ \\
 $(4)$ & $pq$, $p\equiv 4,7\,(9)$, $q\equiv 2,5\,(9)$ & $2$ & $1$ & $2$ & $1$ & $0$ & $1$ & $1$ & $1$ \\
\hline
\end{tabular}
\end{center}
\end{table}


Finally, Table \ref{tbl:BWB01} lists further conductors $f$ with a small number of prime divisors
such that none of them splits in $k_0$, that is $\tilde{s}=0$.


\renewcommand{\arraystretch}{1.1}

\begin{table}[ht]
\caption{Pure cubic fields $L$ with $\tilde{s}=0$ and $r=1$}
\label{tbl:BWB01}
\begin{center}
\begin{tabular}{|cl|ccccc|cc|c|}
\hline
 Item  & $f$  & $\tilde{t}$ & $\tilde{s}$ & $\tilde{v}$ & $\tilde{\varepsilon}$ & $\tilde{\delta}$ & $\max(\tilde{s},\tilde{\delta})$ & $\tilde{s}+\tilde{\delta}$ & $r$ \\
\hline
 $(1)$ & $9q$, $q\equiv 8\,(9)$                                    & $2$ & $0$ & $0$        & $0$ & $1$ & $1$ & $1$ & $1$  \\
 $(2)$ & $q_1q_2$, $q_1,q_2\equiv 8\,(9)$                          & $2$ & $0$ & $0$        & $0$ & $1$ & $1$ & $1$ & $1$  \\
 $(3)$ & $3q_1q_2$, $q_1,q_2\equiv 2,5\,(9)$                       & $3$ & $0$ & $2$        & $1$ & $1$ & $1$ & $1$ & $1$  \\
 $(4)$ & $3q_1q_2$, $q_1\equiv 2,5\,(9)$, $q_2\equiv 8\,(9)$       & $3$ & $0$ & $1$        & $1$ & $1$ & $1$ & $1$ & $1$  \\
 $(5)$ & $9q_1q_2$, $q_1\equiv 2,5\,(9)$, $q_2\equiv 2\,(3)$       & $3$ & $0$ & $1$ or $2$ & $1$ & $1$ & $1$ & $1$ & $1$  \\
 $(6)$ & $q_1q_2q_3$, $q_1,q_2,q_3\equiv 2,5\,(9)$                 & $3$ & $0$ & $3$        & $1$ & $1$ & $1$ & $1$ & $1$  \\
 $(7)$ & $q_1q_2q_3$, $q_1,q_2\equiv 2,5\,(9)$, $q_3\equiv 8\,(9)$ & $3$ & $0$ & $2$        & $1$ & $1$ & $1$ & $1$ & $1$  \\
\hline

\hline
\end{tabular}
\end{center}
\end{table}



\subsection{Class groups with 3-rank zero}
\label{ss:Rank0}

To enable comparison with our main result,
we mention the related earlier result by Taira Honda \cite{Ho}
on ambiguous $3$-class groups having the minimal 
$\operatorname{rank}\,(C_{k,3}^{(\sigma)})=0$.
We extend Honda's theorem by information on the \textit{multiplicity} $m(f)$ of conductors $f$
and on the Galois cohomology of the unit group $E_{k}$,
expressed by the \textit{principal factorization type} (PFT), briefly called the \textit{type}.
(Partially, the following theorem is due to Barrucand and Cohn,
\cite[Cor. 4.2.1, pp. 14--15]{BC1970} and \cite[Cor. 14.1.1, p. 232]{BC1971}.
It is also mentioned by Gerth \cite[Cases 1--2, p. 473]{Ge2005}.)
Honda's result can also be obtained as a consequence of \cite[Thm. 3, p. 399]{Wa} by H. Wada.


\begin{theorem}
\label{thm:Honda}
Let the conductor of $k/k_{0}$ be $f=3^e\cdot\ell_1\cdots\ell_n$
with $0\le e\le 2$, pairwise distinct primes $\ell_j\neq 3$ for $1\le j\le n$,
and $n\ge 0$ if $e=2$, but $n\ge 1$ if $e\le 1$.
Denote the multiplicity of $f$ by $m:=m(f)$.
Then, $3\nmid\# C_{k,3}^{(\sigma)}$
$\Longleftrightarrow$ $3\nmid h_{k}$ 
$\Longleftrightarrow$ $3\nmid h_{L}$
$\Longleftrightarrow$ $L$ belongs to one of the following multiplets,
where always $\ell_j=q_j\equiv -1\,(\mathrm{mod}\,3)$:

\begin{enumerate}
\item[$(1)$]
singulet with $m=1$ of type $\gamma$
such that $f=3^2$,
\item[$(2)$]
singulets with $m=1$ of type $\gamma$
such that $f=q_1$ with $q_1\equiv 8\,(\mathrm{mod}\,9)$,
\item[$(3)$]
singulets with $m=1$ of type $\beta$
such that $f=3q_1$ with $q_1\equiv 2,5\,(\mathrm{mod}\,9)$,
\item[$(4)$]
doublets with $m=2$ of type $(\beta,\beta)$
such that $f=3^2q_1$ with $q_1\equiv 2,5\,(\mathrm{mod}\,9)$,
\item[$(5)$]
singulets with $m=1$ of type $\beta$
such that $f=q_1q_2$ with $q_j\equiv 2,5\,(\mathrm{mod}\,9)$ for $1\le j\le 2$.
\end{enumerate}

\noindent
The singulet $(1)$ is unique,
but there exist infinitely many multiplets of each shape $(2)$--$(5)$.
\end{theorem}


\begin{proof}
We have $3\nmid\# C_{k,3}^{(\sigma)}$ $\Longleftrightarrow$ $3\nmid h_{k}$
by \cite[Lem. 2, p. 7]{Ho}, and the equivalence
$3\nmid h_{k}$ $\Longleftrightarrow$ $3\nmid h_{L}$
is due to the class number formula $h_{k}=\frac{Q}{3}\cdot h_{L}^2$ \cite[Lem. 1, p. 7]{Ho},
where $Q\in\lbrace 1,3\rbrace$, and here necessarily $Q=3$.
The final equivalence follows from Table \ref{tbl:BWB00} in \S\ \ref{ss:RankFormulas}.

Items (1) and (4) correspond to case (i) and (iii) in Honda's paper \cite[Thm., \S\ 1, p. 8]{Ho}.
However, we must split Honda's case (ii) into items (2) and (3),
because the PF types of the singulets are different.
Our item (5) unifies Honda's cases (iv) and (v),
since the conductor $f$ is independent of
the exponents $e_j$ of the prime factors $q_j$ occurring in the radicand $d$.

The multiplicity $m(f)$ of each conductor is calculated by means of 
\cite[Thm. 2.1, p. 833]{Ma1992},
which is exactly our Theorem \ref{thm:Multiplicity},
using the sequence $(X_k)_{k\ge -1}=(\frac{1}{2},0,1,1,3,\ldots)$:
\begin{enumerate}
\item[$(1)$]
For the unique conductor $f=3^2$ of species 1a with $n=0$, we have $m(f)=2^n=1$,
a singulet with prime radicand $d=3$.
\item[$(2)$]
For $f=q_1\equiv 8\,(\mathrm{mod}\,9)$ of species 2, we must take into consideration that $u=1$, $v=0$,
and we obtain $m(f)=2^u\cdot X_{v-1}=2\cdot\frac{1}{2}=1$, a singulet with prime radicand $d=q_1$.
\item[$(3)$]
For $f=3q_1$ of species 1b with $q_1\equiv 2,5\,(\mathrm{mod}\,9)$,
we have $u=0$, $v=1$, $m(f)=2^u\cdot X_{v}=1\cdot 1=1$, a singulet with prime radicand $d=q_1$.
\item[$(4)$]
For $f=3^2q_1$ of species 1a with $n=1$, we get $m(f)=2^n=2$ (independently of $u$ and $v$),
a \textit{doublet} with two associated composite radicands $d=3q_1$ and $d=3^2q_1$.
\item[$(5)$]
For $f=q_1q_2$ of species 2 with $q_1,q_2\equiv 2,5\,(\mathrm{mod}\,9)$,
we have $u=0$, $v=2$, $m(f)=2^u\cdot X_{v-1}=1\cdot 1=1$,
a singulet with composite radicand either $d=q_1q_2$ or $d=q_1^2q_2$.
\end{enumerate}

The principal factorization type,
as a refinement of the Galois cohomology of the unit group $E_{k}$,
is a consequence of the estimates in Theorem \ref{thm:PrincipalFactorization}.
Since $s=0$, we have $0\le R\le 0$ and type $\alpha$ is generally impossible.
For items $(1)$ and $(2)$, we have $t=1$ and thus $1\le A\le 1-0=1$, which discourages type $\beta$.
For all other cases, there exists a prime factor $q_1\equiv 2,5\,(\mathrm{mod}\,9)$,
and thus type $\gamma$ is impossible,
because $\zeta_3$ can be norm of a unit in $k$ only if the prime factors of $f$
are $3$ or $\ell_j\equiv 1,8\,(\mathrm{mod}\,9)$.
This is our new proof of \cite[Thm. 15.7, p. 236]{BC1971}.

All claims on the infinitude of the various sets of conductors $f$
are consequences of Dirichlet's theorem on primes $q\in r+9\mathbb{Z}$ in arithmetic progressions,
here: invertible residue classes $r$ modulo $9$ with $\gcd(r,9)=1$.
\end{proof}


\begin{example}
\label{exm:Honda}
Among the $827\,600$ pure cubic fields $L=\mathbb{Q}(\sqrt[3]{d})$ with normalized radicands $d<10^6$,
there are $73\,885$, that is $8.93\%$, whose class number is not divisible by $3$.
Table \ref{tbl:Honda} shows the contribution (absolute and relative frequency)
of each item in Theorem \ref{thm:Honda} together with all paradigms $d<100$.
Due to the cut off at $d=10^6$, only $3\,519$ of the doublets in item $(4)$ are complete,
the other $6\,060$ pseudo-singulets have companion radicands outside the range of investigations.
Consequently, we have $2\cdot 3\,519+6\,060=13\,098$.
\end{example}


\renewcommand{\arraystretch}{1.0}

\begin{table}[ht]
\caption{$73\,885$ pure cubic fields $L$ with $3\nmid h_{L}$}
\label{tbl:Honda}
\begin{center}
\begin{tabular}{|cl|c|rr|l|}
\hline
 Item  & $f$                            & Type            & $\#$      & $\%$    & Paradigms for $d$          \\
\hline
 $(1)$ & $9$                            & $\gamma$        &       $1$ &  $0.00$ & $3$                        \\
 $(2)$ & $q\equiv 8\,(9)$               & $\gamma$        & $13\,099$ & $17.73$ & $17,53,71,89$              \\
 $(3)$ & $3q$, $q\equiv 2,5\,(9)$       & $\beta$         & $26\,167$ & $35.42$ & $2,5,11,23,29,41,47,59,83$ \\
 $(4)$ & $9q$, $q\equiv 2,5\,(9)$       & $(\beta,\beta)$ & $13\,098$ & $17.73$ & $6,12,15,33,45,69,87,99$   \\
 $(5)$ & $q_1q_2$, $q_j\equiv 2,5\,(9)$ & $\beta$         & $21\,520$ & $29.13$ & $10,44,46,55,82$           \\
\hline
\end{tabular}
\end{center}
\end{table}



\subsection{Conductors divisible by a splitting prime}
\label{ss:SplitPrime}

Now we come to ambiguous class groups $C_{k,3}^{(\sigma)}$ of $3$-rank one,
and we first give more details concerning the leading three lines
of equation \eqref{eqn:Rank1} in our Theorem \ref{thm:Rank1},
where $d$ is divisible by a prime $p_1\equiv 1\,(\mathrm{mod}\,3)$
which splits in $k_0$.

\begin{theorem}
\label{thm:Ismaili1}
Let the conductor of $k/k_{0}$ be $f=3^e\cdot\ell_1\cdots\ell_n$
with $0\le e\le 2$, $n\ge 1$, and pairwise distinct primes $\ell_j\neq 3$ for $1\le j\le n$.
Briefly denote the multiplicity of $f$ by $m:=m(f)$.
Assume that $\ell_j\equiv 1\,(\mathrm{mod}\,3)$ for at least one $1\le j\le n$.
Then, $\operatorname{rank}\,(C_{k,3}^{(\sigma)})=1$
$\Longleftrightarrow$ $L$ belongs to one of the following multiplets,
where $\ell_1=p_1\equiv 1\,(\mathrm{mod}\,3)$ and $\ell_2=q_2\equiv -1\,(\mathrm{mod}\,3)$.

\begin{enumerate}
\item[$(1)$]
singulets with $m=1$ of type $\alpha$ or $\gamma$
such that $f=p_1$ with $p_1\equiv 1\,(\mathrm{mod}\,9)$,
\item[$(2)$]
singulets with $m=1$ of type $\alpha$ or $\beta$
such that $f=3p_1$ with $p_1\equiv 4,7\,(\mathrm{mod}\,9)$,
\item[$(3)$]
doublets with $m=2$ of type $(\alpha,\alpha)$ or $(\beta,\beta)$
such that $f=3^2p_1$ with $p_1\equiv 4,7\,(\mathrm{mod}\,9)$,
\item[$(4)$]
singulets with $m=1$ of type $\alpha$ or $\beta$
such that $f=p_1q_2$ with $p_1\equiv 4,7\,(\mathrm{mod}\,9)$ and $q_2\equiv 2,5\,(\mathrm{mod}\,9)$.
\end{enumerate}

\noindent
There exist infinitely many multiplets with conductors of all these shapes $(1)$--$(4)$.
For each of these conductors,
the ambiguous $3$-class group $C_{k,3}^{(\sigma)}\simeq (3)$
of the normal closure $k$ of $L$
is cyclic of order $3$.
\end{theorem}


\begin{proof}
Given the assumption that $\ell_j\equiv 1\,(\mathrm{mod}\,3)$ for at least one $1\le j\le n$,
the equivalence of the rank condition $\operatorname{rank}\,(C_{k,3}^{(\sigma)})=1$
to the four shapes of conductors will be proved in \S\S\ \ref{ss:SingleSplitPrime}--\ref{ss:SplitAndNonSplitPrime}
and in the Cases (1)--(4) of \S\ \ref{ss:ProofRank1}.

The multiplicity $m(f)$ of each conductor is calculated by means of 
\cite[Thm. 2.1, p. 833]{Ma1992},
that is our Theorem \ref{thm:Multiplicity},
using the sequence $(X_k)_{k\ge -1}=(\frac{1}{2},0,1,1,3,\ldots)$:
\begin{enumerate}
\item[$(1)$]
For $f=p_1\equiv 1\,(\mathrm{mod}\,9)$ of species 2, we must take into consideration that $u=1$, $v=0$,
and we obtain $m(f)=2^u\cdot X_{v-1}=2\cdot\frac{1}{2}=1$, a singulet with prime radicand $d=p_1$.
\item[$(2)$]
For $f=3p_1$ of species 1b with $p_1\equiv 4,7\,(\mathrm{mod}\,9)$,
we have $u=0$, $v=1$, $m(f)=2^u\cdot X_{v}=1\cdot 1=1$, a singulet with prime radicand $d=p_1$.
\item[$(3)$]
For $f=3^2p_1$ of species 1a with $n=1$, we get $m(f)=2^n=2$ (independently of $u$ and $v$),
a \textit{doublet} with two associated composite radicands $d=3p_1$ and $d=3^2p_1$.
\item[$(4)$]
For $f=p_1q_2$ of species 2 with $p_1,-q_2\equiv 4,7\,(\mathrm{mod}\,9)$,
we have $u=0$, $v=2$, $m(f)=2^u\cdot X_{v-1}=1\cdot 1=1$,
a singulet with composite radicand either $d=p_1q_2$ or $d=p_1^2q_2$.
\end{enumerate}

The principal factorization type
is a consequence of the estimates in Theorem \ref{thm:PrincipalFactorization}.
Since $s=1$, we have $0\le R\le 1$ and type $\alpha$ may generally occur.
For item $(1)$, we have $t=2$ and thus $1\le A\le 2-1=1$, which discourages type $\beta$,
but type $\gamma$ may occur.
For items $(2)$--$(4)$, there exists a prime factor $p_1\equiv 4,7\,(\mathrm{mod}\,9)$,
and type $\gamma$ is impossible,
because $\zeta_3$ can be norm of a unit in $k$ only if the prime factors of $f$
are $3$ or $\ell_j\equiv 1,8\,(\mathrm{mod}\,9)$.

All claims on the infinitude of the various sets of conductors $f$
are a consequence of Dirichlet's theorem on primes in arithmetic progressions.
\end{proof}



\begin{corollary}
\label{cor:Ismaili1}
For each of the pure cubic fields $L$ belonging to
one of the items $(2)$, $(3)$, $(4)$ in Theorem \ref{thm:Ismaili1},
the following equivalences specify the correlation between
types, subfield unit indices $Q$,
$3$-class group structures of the normal closure $k$,
and cubic residue symbols.
\begin{enumerate}
\item
$L$ is of type $\alpha$ $\Longleftrightarrow$ $Q=1$, and
$\biggl(Q=1,\ C_{k,3}\simeq (3)\biggr)$
$\Longleftrightarrow$  $\left(\frac{c}{p_1}\right)_3\neq 1$,
\item
$L$ is of type $\beta$ $\Longleftrightarrow$ $Q=3$, and
$\biggl(Q=3,\ C_{k,3}\simeq (3,3)\biggr)$
$\Longrightarrow$ $\biggl\lbrack\stackrel{conjecture}{\Longleftrightarrow}\biggr\rbrack$ $\left(\frac{c}{p_1}\right)_3=1$,
\end{enumerate}
where $c=3$ in items $(2)$, $(3)$, and $c=q_2$ in item $(4)$.
For each of these conductors,
the $3$-class group $C_{L,3}\simeq (3)$ of $L$
is cyclic of order $3$.
If $C_{k,3}\simeq (3,3)$ is elementary bicyclic, then the fields
$k^{\ast}=k\cdot L_1=k\cdot L_1^{\sigma}=k\cdot L_1^{\sigma^2}=K_4$
coincide with one of the four unramified cyclic cubic extensions $K_1,\ldots,K_4$ of $k$ within $k_1$,
as illustrated in Figure \ref{fig:IsmailiType1}
(valid for $L$ of type $\beta$).
\end{corollary}



\begin{figure}[h]
\caption{First possible location of the relative $3$-genus field $k^{\ast}=(k/k_0)^{\ast}$}
\label{fig:IsmailiType1}

{\tiny

\setlength{\unitlength}{1.0cm}
\begin{picture}(15,11)(-11,-10)

\put(-10,0.5){\makebox(0,0)[cb]{Degree}}

\put(-10,-2){\vector(0,1){2}}

\put(-10,-2){\line(0,-1){7}}
\multiput(-10.1,-2)(0,-1){8}{\line(1,0){0.2}}

\put(-10.2,-2){\makebox(0,0)[rc]{\(54\)}}
\put(-9.8,-2){\makebox(0,0)[lc]{Hilbert of bicubic}}
\put(-10.2,-4){\makebox(0,0)[rc]{\(18\)}}
\put(-9.8,-4){\makebox(0,0)[lc]{compositum}}
\put(-10.2,-5){\makebox(0,0)[rc]{\(9\)}}
\put(-9.8,-5){\makebox(0,0)[lc]{Hilbert of conjugate cubics}}
\put(-10.2,-6){\makebox(0,0)[rc]{\(6\)}}
\put(-9.8,-6){\makebox(0,0)[lc]{bicubic}}
\put(-10.2,-7){\makebox(0,0)[rc]{\(3\)}}
\put(-9.8,-7){\makebox(0,0)[lc]{conjugate cubics}}
\put(-10.2,-8){\makebox(0,0)[rc]{\(2\)}}
\put(-9.8,-8){\makebox(0,0)[lc]{quadratic}}
\put(-10.2,-9){\makebox(0,0)[rc]{\(1\)}}
\put(-9.8,-9){\makebox(0,0)[lc]{base}}

{\normalsize
\put(-3,0){\makebox(0,0)[cc]{Scenario I: \quad \(k^{\ast}=L_1\cdot k\)}}
}



\put(0,-2){\circle*{0.2}}
\put(0.2,-2){\makebox(0,0)[lc]{\(k_1\)}}

\put(0,-2){\line(-3,-2){3}}
\put(0,-2){\line(-1,-2){1}}
\put(0,-2){\line(1,-2){1}}
\put(0,-2){\line(3,-2){3}}

\multiput(-3,-4)(2,0){4}{\circle*{0.2}}
\put(-3.2,-4){\makebox(0,0)[rc]{\(k^{\ast}=L_1\cdot k=L_1^{\sigma}\cdot k=L_1^{\sigma^2}\cdot k\)}}
\put(-1.2,-4){\makebox(0,0)[rc]{\(K_3\)}}
\put(1.2,-4){\makebox(0,0)[lc]{\(K_2\)}}
\put(3.2,-4){\makebox(0,0)[lc]{\(K_1\)}}

\put(0,-6){\line(-3,2){3}}
\put(0,-6){\line(-1,2){1}}
\put(0,-6){\line(1,2){1}}
\put(0,-6){\line(3,2){3}}

\put(0,-6){\circle*{0.2}}
\put(0.2,-6){\makebox(0,0)[lc]{\(k\)}}

\put(0,-6){\line(0,-1){2}}

\put(0,-8){\circle*{0.2}}
\put(0.2,-8){\makebox(0,0)[lc]{\(k_0\)}}

\put(-5,-5){\line(2,1){2}}
\put(-2,-7){\line(2,1){2}}
\put(-2,-9){\line(2,1){2}}


\put(-5,-5){\circle{0.2}}
\put(-5.2,-5){\makebox(0,0)[rc]{\(L_1\)}}
\put(-4.8,-5){\makebox(0,0)[lc]{\(L_1^{\sigma},L_1^{\sigma^2}\)}}

\put(-2,-7){\line(-3,2){3}}

\put(-2,-7){\circle{0.2}}
\put(-2.2,-7){\makebox(0,0)[rc]{\(L\)}}
\put(-1.8,-7){\makebox(0,0)[lc]{\(L^{\sigma},L^{\sigma^2}\)}}

\put(-2,-7){\line(0,-1){2}}

\put(-2,-9){\circle*{0.2}}
\put(-2.2,-9){\makebox(0,0)[rc]{\(\mathbb{Q}\)}}


\end{picture}

}

\end{figure}
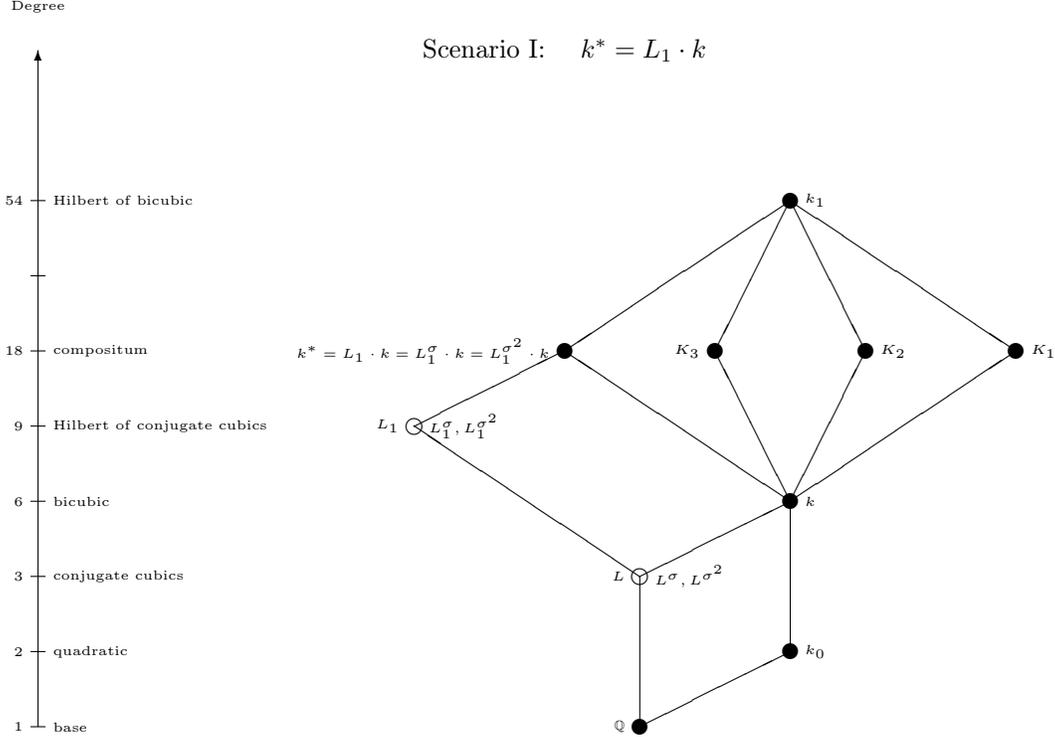


\begin{proof}
The claims for item $(2)$ have been shown partially by Gerth
\cite[Case 3, pp. 474--475]{Ge2005}.
For all items $(2)$, $(3)$, $(4)$, they are due to Ismaili and El Mesaoudi
\cite[Thm. 3.2, p. 104]{IsEM2}.
The statement concerning the $3$-genus field $k^{\ast}$ was proved by Ismaili
\cite[Thm. 3.2, pp. 29--31]{Is}.
We also emphasize the connection with Table \ref{tbl:BWB11} in \S\ \ref{ss:RankFormulas}.
\end{proof}



\begin{corollary}
\label{cor:Gerth1Mod9}
For each of the pure cubic fields $L$ belonging to
item $(1)$ in Theorem \ref{thm:Ismaili1},
the following equivalences specify the correlation between
types, subfield unit indices $Q$, and $3$-class group structures of $L$ and $k$.
\begin{enumerate}
\item
$L$ is of type $\alpha$ $\Longleftrightarrow$ $Q=1$ $\Longleftrightarrow$
$C_{L,3}\simeq (3^w)$ and $C_{k,3}\simeq (3^w,3^{w-1})$ for some $w\ge 1$,
\item
$L$ is of type $\gamma$ $\Longleftrightarrow$ $Q=3$ $\Longleftrightarrow$
$C_{L,3}\simeq (3^w)$ and $C_{k,3}\simeq (3^w,3^w)$ for some $w\ge 2$.
\end{enumerate}
Here, the $3$-class group of $k$ is either heterocyclic if $Q=1$
(in particular, $(3^w,3^{w-1})$ is to be interpreted as the cyclic type $(3)$ for $w=1$)
or homocyclic if $Q=3$, but never of elementary bicyclic type $(3,3)$.
Also, $Q=1$ $\Longleftrightarrow$ $\mathfrak{P}\in\mathcal{P}_{k}$,
where $p_1\mathcal{O}_{k}=\mathfrak{P}^3(\mathfrak{P}^{\tau})^3$ ($\mathfrak{P}$ a prime ideal of $k$).
\end{corollary}


\begin{proof}
These statements have been proved by Gerth in
\cite[Formulas p. 474, and Case 4, pp. 475--476]{Ge2005}.
\end{proof}


\renewcommand{\arraystretch}{1.0}

\begin{table}[ht]
\caption{$81\,894$ pure cubic fields $L$ with $s=1$, $\operatorname{rank}\,(C_{k,3}^{(\sigma)})=1$}
\label{tbl:Ismaili1}
\begin{center}
\begin{tabular}{|cl|c|rr|l|}
\hline
 Item  & $f$                                          & Type              & $\#$      & $\%$    & Paradigms for $d$          \\
\hline
 $(1)$ & $p_1\equiv 1\,(9)$                           &                   & $13\,063$ & $15.95$ &                            \\
       &                                              & $\alpha$          & $11\,958$ & $91.54$ & $19,37,73$                 \\
       &                                              & $\gamma$          &  $1\,105$ &  $8.46$ & $541,919,1279$             \\
\hline
 $(2)$ & $3p_1$, $p_1\equiv 4,7\,(9)$                 &                   & $26\,168$ & $31.95$ &                            \\
       &                                              & $\alpha$          & $17\,485$ & $66.82$ & $7,13,31,43,79,97$         \\
       &                                              & $\beta$           &  $8\,683$ & $33.18$ & $61,67,103,151$            \\
\hline
 $(3)$ & $9p_1$, $p_1\equiv 4,7\,(9)$                 &                   & $13\,048$ & $15.93$ &                            \\
       &                                              & $(\alpha,\alpha)$ &  $8\,709$ & $66.75$ & $21,39,63,93$              \\
       &                                              & $(\beta,\beta)$   &  $4\,339$ & $33.25$ & $183,201,309,453$          \\
\hline
 $(4)$ & $p_1q_1$, $p_1\equiv 4,7\,(9)$, $q_1\equiv 2,5\,(9)$ &           & $29\,615$ & $36.16$ &                            \\
       &                                              & $\alpha$          & $19\,898$ & $67.19$ & $26,28,35$                 \\
       &                                              & $\beta$           &  $9\,717$ & $32.81$ & $62,172,287$               \\
\hline
\end{tabular}
\end{center}
\end{table}


\begin{example}
\label{exm:Ismaili1}
Among the $827\,600$ pure cubic fields $L=\mathbb{Q}(\sqrt[3]{d})$ with normalized radicands $d<10^6$,
there are $81\,894$, that is $9.90\%$, with $s=1$ and $\operatorname{rank}\,(C_{k,3}^{(\sigma)})=1$.
Table \ref{tbl:Ismaili1} shows the contribution (absolute and relative frequency)
of each item in Theorem \ref{thm:Ismaili1} together with all paradigms $d<100$ or even bigger.
Here, we have to take into account that different PF types are possible.
Due to the cut off at $d=10^6$, only $3\,514$ of the doublets in item $(3)$ are complete,
the other $6\,020$ pseudo-singulets have companion radicands outside the range of investigations.
Consequently, we have $2\cdot 3\,514+6\,020=13\,048=8\,709+4\,339$.
Here, $2\,348$ doublets are of type $(\alpha,\alpha)$ and $1\,166$ of type $(\beta,\beta)$,
but inhomogeneous doublet types $(\alpha,\beta)$ do not occur, which could be explained
when the implication for cubic residue symbols in Corollary \ref{cor:Ismaili1} were an equivalence.
Among the pseudo-singulets, $4\,013$ are of type $\alpha$ and $2\,007$ of type $\beta$.
The exponent $w$ in Corollary \ref{cor:Gerth1Mod9} seems to be unbounded:
in addition to Table \ref{tbl:Ismaili1}, we mention that
for type $\alpha$,
occurrences of $w=2$ set in with $d\in\lbrace 199,271,487\rbrace$,
$w=3$ with $d\in\lbrace 3061,3583,4177\rbrace$,
and $w=4$ with $d\in\lbrace 6733,8263\rbrace$,
for type $\gamma$ we have $w=3$ for $d=8389$.
\end{example}



\subsection{Conductors divisible by non-split primes only}
\label{ss:NonSplitPrimes}

\noindent
We continue the discussion of ambiguous class groups $C_{k,3}^{(\sigma)}$ of $3$-rank one
by giving more details about the trailing six lines
of equation \eqref{eqn:Rank1} in our Theorem \ref{thm:Rank1},
where $d$ is only divisible by primes $q_j\equiv -1\,(\mathrm{mod}\,3)$
which do not split in $k_0$.


\begin{theorem}
\label{thm:Ismaili2}
Let the conductor of $k/k_{0}$ be $f=3^e\cdot\ell_1\cdots\ell_n$
with $0\le e\le 2$, $n\ge 1$, and pairwise distinct primes $\ell_j\neq 3$ for $1\le j\le n$.
Denote the multiplicity of $f$ by $m:=m(f)$.
Assume that $\ell_j\equiv -1\,(\mathrm{mod}\,3)$ for all $1\le j\le n$.
Then, $\operatorname{rank}\,(C_{k,3}^{(\sigma)})=1$
$\Longleftrightarrow$ $L$ belongs to one of the following multiplets,
where $\ell_j=q_j\equiv -1\,(\mathrm{mod}\,3)$ for $1\le j\le 3$.

\begin{enumerate}
\item[$(1)$]
doublets with $m=2$ of type $(\beta^x,\gamma^y)$, $x+y=2$,
such that $f=3^2q_1$ with $q_1\equiv 8\,(\mathrm{mod}\,9)$,
\item[$(2)$]
doublets with $m=2$ of type $(\beta^x,\gamma^y)$, $x+y=2$,
such that $f=q_1q_2$ with $q_1,q_2\equiv 8\,(\mathrm{mod}\,9)$,
\item[$(3)$]
singulets with $m=1$ of type $\beta$
such that $f=3q_1q_2$ with $q_1,q_2\equiv 2,5\,(\mathrm{mod}\,9)$,
\item[$(4)$]
doublets with $m=2$ of type $(\beta,\beta)$
such that $f=3q_1q_2$ with $q_1\equiv 2,5\,(\mathrm{mod}\,9)$, $q_2\equiv 8\,(\mathrm{mod}\,9)$,
\item[$(5)$]
quartets with $m=4$ of type $(\beta,\beta,\beta,\beta)$
such that $f=3^2q_1q_2$ with $q_1\equiv 2,5\,(\mathrm{mod}\,9)$, $q_2\equiv -1\,(\mathrm{mod}\,3)$,
\item[$(6)$]
singulets with $m=1$ of type $\beta$
such that $f=q_1q_2q_3$ with $q_1,q_2,q_3\equiv 2,5\,(\mathrm{mod}\,9)$.
\item[$(7)$]
doublets with $m=2$ of type $(\beta,\beta)$
such that $f=q_1q_2q_3$ with $q_1,q_2\equiv 2,5\,(\mathrm{mod}\,9)$, $q_3\equiv 8\,(\mathrm{mod}\,9)$,
\end{enumerate}

\noindent
There exist infinitely many multiplets with conductors of all these shapes $(1)$--$(7)$.
For each of these conductors,
the ambiguous $3$-class group $C_{k,3}^{(\sigma)}\simeq (3)$
of the normal closure $k$ of $L$
is cyclic of order $3$.
\end{theorem}


\begin{proof}
Given the assumption that $\ell_j\equiv -1\,(\mathrm{mod}\,3)$ for all $1\le j\le n$,
the equivalence of the rank condition $\operatorname{rank}\,(C_{k,3}^{(\sigma)})=1$
to the seven shapes of conductors will be proved in \S\S\ \ref{ss:SingleNonSplitPrime}--\ref{ss:ThreeNonSplitPrimes}
and in the Cases (5)--(10) of \S\ \ref{ss:ProofRank1}.

The multiplicity $m(f)$ of each conductor is calculated by means of 
\cite[Thm. 2.1, p. 833]{Ma1992}
using the sequence $(X_k)_{k\ge -1}=(\frac{1}{2},0,1,1,3,\ldots)$:
\begin{enumerate}
\item[$(1)$]
For $f=3^2q_1$ of species 1a with $n=1$, we get $m(f)=2^n=2$ (independently of $u$ and $v$),
a \textit{doublet} with two associated companion radicands $d=3q_1$ and $d=3^2q_1$.
\item[$(2)$]
For $f=q_1q_2$ with $q_1,q_2\equiv 8\,(\mathrm{mod}\,9)$ of species 2, we must take into consideration that $u=2$, $v=0$,
and we obtain $m(f)=2^u\cdot X_{v-1}=4\cdot\frac{1}{2}=2$,
a \textit{doublet} with two associated companion radicands $d=q_1q_2$ and $d=q_1^2q_2$.
\item[$(3)$]
For $f=3q_1q_2$ of species 1b with $q_1,q_2\equiv 2,5\,(\mathrm{mod}\,9)$,
we have $u=0$, $v=2$, $m(f)=2^u\cdot X_{v}=1\cdot 1=1$,
a singulet with composite radicand $d=q_1^{e_1}q_2^{e_2}$ satisfying $d\not\equiv\pm 1\,(\mathrm{mod}\,9)$.
\item[$(4)$]
For $f=3q_1q_2$ of species 1b with $q_1\equiv 2,5\,(\mathrm{mod}\,9)$, $q_2\equiv 8\,(\mathrm{mod}\,9)$
we have $u=1$, $v=1$, $m(f)=2^u\cdot X_{v}=2\cdot 1=2$,
a \textit{doublet} with two associated companion radicands of the shape $d=q_1^{e_1}q_2^{e_2}$ satisfying $d\not\equiv\pm 1\,(\mathrm{mod}\,9)$.
\item[$(5)$]
For $f=3^2q_1q_2$ of species 1a with $n=2$, we get $m(f)=2^n=4$ (independently of $u$ and $v$),
a \textit{quartet} with four associated companion radicands $d=3q_1q_2$, $d=3^2q_1q_2$, $d=3q_1^2q_2$, $d=3q_1q_2^2$.
\item[$(6)$]
For $f=q_1q_2q_3$ of species 2 with $q_1,q_2,q_3\equiv 2,5\,(\mathrm{mod}\,9)$,
we have $u=0$, $v=3$, $m(f)=2^u\cdot X_{v-1}=1\cdot 1=1$,
a singulet with composite radicand $d=q_1^{e_1}q_2^{e_2}q_3^{e_3}$ satisfying $d\equiv\pm 1\,(\mathrm{mod}\,9)$.
\item[$(7)$]
For $f=q_1q_2q_3$ of species 2 with $q_1,q_2\equiv 2,5\,(\mathrm{mod}\,9)$, $q_3\equiv 8\,(\mathrm{mod}\,9)$
we have $u=1$, $v=2$, $m(f)=2^u\cdot X_{v-1}=2\cdot 1=2$,
a \textit{doublet} with two associated companion radicands of the shape $d=q_1^{e_1}q_2^{e_2}q_3^{e_3}$
satisfying $d\equiv\pm 1\,(\mathrm{mod}\,9)$.
\end{enumerate}

The principal factorization type
is a consequence of the estimates in Theorem \ref{thm:PrincipalFactorization}.
Since $s=0$, we have $0\le R\le 0$ and type $\alpha$ is generally forbidden.
For all cases, we have $t\ge 2$ and thus $1\le A\le 2-0=2$, which enables type $\beta$.
For items $(3)$--$(7)$, there exists a prime factor $q_1\equiv 2,5\,(\mathrm{mod}\,9)$,
and type $\gamma$ is impossible,
because $\zeta_3$ can be norm of a unit in $k$ only if the prime factors of $f$
are $3$ or $\ell_j\equiv 1,8\,(\mathrm{mod}\,9)$.
The latter condition is satisfied by items $(1)$--$(2)$, whence type $\gamma$ may occur.

All claims on the infinitude of the various sets of conductors $f$
are consequences of Dirichlet's theorem on prime numbers arising from invertible residue classes.
\end{proof}


\begin{corollary}
\label{cor:Ismaili2}
For each of the pure cubic fields $L$ belonging to
one of the items $(1)$--$(7)$ in Theorem \ref{thm:Ismaili2},
the $3$-class group $C_{L,3}\simeq (3^w)$ of $L$ is cyclic of order $3^w$.
If the $3$-class group $C_{k,3}\simeq (3,3)$ of $k$ is elementary bicyclic, then the four distinct fields
$k^{\ast}$, $k\cdot L_1$, $k\cdot L_1^{\sigma}$, $k\cdot L_1^{\sigma^2}$
are the four unramified cyclic cubic extensions of $k$ within $k_1$,
as illustrated in Figure \ref{fig:IsmailiType2}.
\end{corollary}


\begin{proof}
The statement concerning the $3$-genus field $k^{\ast}$ has been proved by Ismaili
\cite[Thm. 3.2, pp. 29--31]{Is}. The connection with Table \ref{tbl:BWB01} in \S\ \ref{ss:RankFormulas} should be emphasized.
\end{proof}



\begin{figure}[ht]
\caption{Second possible location of the relative $3$-genus field $k^{\ast}=(k/k_0)^{\ast}$}
\label{fig:IsmailiType2}
{\tiny

\setlength{\unitlength}{1.0cm}
\begin{picture}(15,11)(-11,-10)

\put(-10,0.5){\makebox(0,0)[cb]{Degree}}

\put(-10,-2){\vector(0,1){2}}

\put(-10,-2){\line(0,-1){7}}
\multiput(-10.1,-2)(0,-1){8}{\line(1,0){0.2}}

\put(-10.2,-2){\makebox(0,0)[rc]{\(54\)}}
\put(-9.8,-2){\makebox(0,0)[lc]{Hilbert of bicubic}}
\put(-10.2,-4){\makebox(0,0)[rc]{\(18\)}}
\put(-9.8,-4){\makebox(0,0)[lc]{composita}}
\put(-10.2,-5){\makebox(0,0)[rc]{\(9\)}}
\put(-9.8,-5){\makebox(0,0)[lc]{Hilbert of conjugate cubics}}
\put(-10.2,-6){\makebox(0,0)[rc]{\(6\)}}
\put(-9.8,-6){\makebox(0,0)[lc]{bicubic}}
\put(-10.2,-7){\makebox(0,0)[rc]{\(3\)}}
\put(-9.8,-7){\makebox(0,0)[lc]{conjugate cubics}}
\put(-10.2,-8){\makebox(0,0)[rc]{\(2\)}}
\put(-9.8,-8){\makebox(0,0)[lc]{quadratic}}
\put(-10.2,-9){\makebox(0,0)[rc]{\(1\)}}
\put(-9.8,-9){\makebox(0,0)[lc]{base}}

{\normalsize
\put(-3,0){\makebox(0,0)[cc]{Scenario II: \quad \(k^{\ast}\neq L_1\cdot k\), \quad \(k^{\ast}\neq k_1\)}}
}



\put(0,-2){\circle*{0.2}}
\put(0.2,-2){\makebox(0,0)[lc]{\(k_1\)}}

\put(0,-2){\line(-3,-2){3}}
\put(0,-2){\line(-1,-2){1}}
\put(0,-2){\line(1,-2){1}}
\put(0,-2){\line(3,-2){3}}

\multiput(-3,-4)(2,0){4}{\circle*{0.2}}
\put(-3.2,-4){\makebox(0,0)[rc]{\(L_1\cdot k\)}}
\put(-1.2,-4){\makebox(0,0)[rc]{\(L_1^{\sigma}\cdot k\)}}
\put(1.2,-4){\makebox(0,0)[lc]{\(L_1^{\sigma^2}\cdot k\)}}
\put(3.2,-4){\makebox(0,0)[lc]{\(k^{\ast}\)}}

\put(0,-6){\line(-3,2){3}}
\put(0,-6){\line(-1,2){1}}
\put(0,-6){\line(1,2){1}}
\put(0,-6){\line(3,2){3}}

\put(0,-6){\circle*{0.2}}
\put(0.2,-6){\makebox(0,0)[lc]{\(k\)}}

\put(0,-6){\line(0,-1){2}}

\put(0,-8){\circle*{0.2}}
\put(0.2,-8){\makebox(0,0)[lc]{\(k_0\)}}

\put(-5,-5){\line(2,1){2}}
\put(-5,-5){\line(4,1){4}}
\put(-5,-5){\line(6,1){6}}
\put(-2,-7){\line(2,1){2}}
\put(-2,-9){\line(2,1){2}}


\put(-5,-5){\circle{0.2}}
\put(-5.2,-5){\makebox(0,0)[rt]{\(L_1\)}}
\put(-4.6,-5){\makebox(0,0)[lt]{\(L_1^{\sigma},L_1^{\sigma^2}\)}}

\put(-2,-7){\line(-3,2){3}}

\put(-2,-7){\circle{0.2}}
\put(-2.2,-7){\makebox(0,0)[rc]{\(L\)}}
\put(-1.8,-7){\makebox(0,0)[lc]{\(L^{\sigma},L^{\sigma^2}\)}}

\put(-2,-7){\line(0,-1){2}}

\put(-2,-9){\circle*{0.2}}
\put(-2.2,-9){\makebox(0,0)[rc]{\(\mathbb{Q}\)}}


\end{picture}

}

\end{figure}
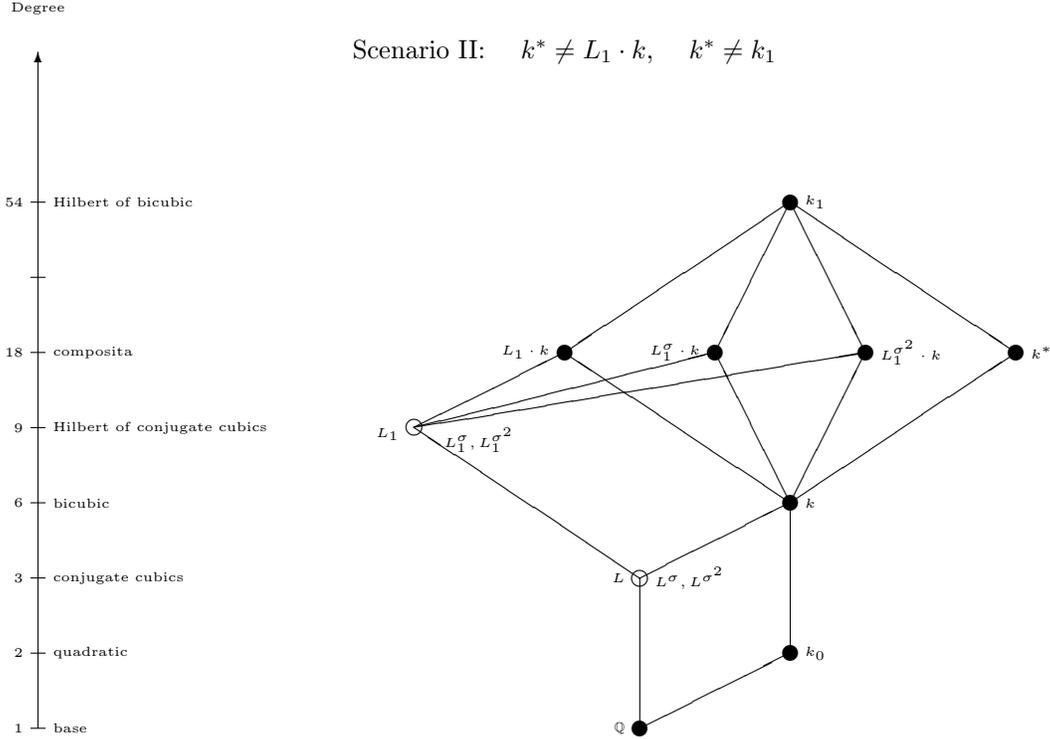



\renewcommand{\arraystretch}{1.0}

\begin{table}[ht]
\caption{$107\,595$ pure cubic fields $L$ with $s=0$, $\operatorname{rank}\,(C_{k,3}^{(\sigma)})=1$}
\label{tbl:Ismaili2}
\begin{center}
\begin{tabular}{|cl|c|rr|l|}
\hline
 Item  & $f$                                                        & Type              & $\#$      & $\%$    & Paradigms for $d$          \\
\hline
 $(1)$ & $9q_1$, $q_1\equiv 8\,(9)$                                 &                   &  $6\,538$ &  $6.08$ &                            \\
       &                                                            & $(\beta,\ast)$    &  $4\,835$ & $73.95$ & $51,159,213$               \\
       &                                                            & $(\gamma,\ast)$   &  $1\,703$ & $26.05$ & $153,321,477$              \\
\hline
 $(2)$ & $q_1q_2$, $q_1,q_2\equiv 8\,(9)$                           &                   &  $3\,007$ &  $2.79$ &                            \\
       &                                                            & $(\beta,\ast)$    &  $2\,259$ & $75.12$ & $901,1819,3043$            \\
       &                                                            & $(\gamma,\ast)$   &     $748$ & $24.88$ & $1207,1513,3763$           \\
\hline
 $(3)$ & $3q_1q_2$, $q_1,q_2\equiv 2,5\,(9)$                        & $\beta$           & $21\,460$ & $19.95$ & $20,22,58,92,94$           \\
 $(4)$ & $3q_1q_2$, $q_1\equiv 2,5\,(9)$, $q_2\equiv 8\,(9)$        & $(\beta,\beta)$   & $27\,510$ & $25.57$ & $34,68,85,106$             \\
\hline
 $(5)$ & $9q_1q_2$, $q_1\equiv 2,5\,(9)$, $q_2\equiv 2\,(3)$ & $(\beta,\beta,\beta,\beta)$ & $34\,170$ & $31.76$ &                         \\
       & $9q_1q_2$, $q_1\equiv 2,5\,(9)$, $q_2\equiv 2,5\,(9)$      &                   & $20\,999$ & $61.45$ & $30,60,66,90$              \\
       & $9q_1q_2$, $q_1\equiv 2,5\,(9)$, $q_2\equiv 8\,(9)$        &                   & $13\,171$ & $38.55$ & $102,204,255,306$          \\
\hline
 $(6)$ & $q_1q_2q_3$, $q_1,q_2,q_3\equiv 2,5\,(9)$                  & $\beta$           &  $5\,249$ &  $4.88$ & $460,550,638,820$          \\
 $(7)$ & $q_1q_2q_3$, $q_1,q_2\equiv 2,5\,(9)$, $q_3\equiv 8\,(9)$  & $(\beta,\beta)$   &  $9\,661$ &  $8.98$ & $170,530,710,748$          \\
\hline
\end{tabular}
\end{center}
\end{table}



\begin{example}
\label{exm:Ismaili2}
Among the $827\,600$ pure cubic fields $L=\mathbb{Q}(\sqrt[3]{d})$ with normalized radicands $d<10^6$,
there are $107\,595$, that is $13.00\%$, with $s=0$ and $\operatorname{rank}\,(C_{k,3}^{(\sigma)})=1$.
Table \ref{tbl:Ismaili2} shows the contribution (absolute and relative frequency)
of each item in Theorem \ref{thm:Ismaili2} together with all paradigms $d<100$ or even bigger.
Again, we partially have to take into account that different PF types are possible.
Due to the cut off at $d=10^6$, only $1\,758$ of the doublets in item $(1)$ are complete,
the other $3\,022$ pseudo-singulets have companion radicands outside the range of investigations.
Consequently, we have $2\cdot 1\,758+3\,022=6\,538=4\,835+1\,703$.
Here, $800$ doublets are of type $(\beta,\beta)$, $17$ of type $(\gamma,\gamma)$,
and also $941$ inhomogeneous doublet types $(\beta,\gamma)$ occur.
Among the pseudo-singulets, $2\,294$ are of type $\beta$ and $728$ of type $\gamma$.
\end{example}







\section{Proof of the Main Theorem \ref{thm:Rank1}}
\label{s:Proof}

In this section, we investigate the $3$-rank of the group of ambiguous classes of $k/k_0$.
More results on $3$-class groups of cubic fields can be found in
Gerth \cite{Ge1973}, \cite{Ge1976}, \cite{Ge1975}, \cite{Ge2005}, and Kobayashi \cite{Ky1973}, \cite{Ky1974}, \cite{Ky1977}.
For the prime factorization in pure cubic fields $L=\mathbb{Q}(\sqrt[3]{d})$,
we refer to the papers by Markoff \cite{Mk}, Dedekind \cite{Dk}, and Barrucand/Cohn \cite{BC1970}, \cite{BC1971}.
For the prime decomposition rules in the cyclotomic field $k_0$,
we refer to Ireland/Rosen \cite[Prop. 1.4.2, p. 13, and Prop. 9.1.1--4, pp. 109--111]{IlRo}.
In the following lemmas, we introduce Gerth's convention \cite[\S\ 5, pp. 91--92]{Ge1976}
for characterizing a unique \textit{primary} prime element among its six associates in the ring $\mathcal{O}_{k_0}$,
and we point out that Ireland/Rosen \cite[Prop. 9.3.5, pp. 113--114]{IlRo} use a different convention.


\begin{lemma}
\label{pilamdaP}
Let $p$ be a prime number congruent to $4$ or $7$ mod $9$,
and thus $p=\pi_1\pi_2$ with two prime elements $\pi_1$ and $\pi_2=\pi_1^{\tau}$ of $k_0$.
If $\pi_1\equiv\pi_2\equiv 1(\mathrm{mod}\, 3\mathcal{O}_{k_0})$, then
$\pi_1 \text{ and } \pi_2\not\equiv 1(\mathrm{mod}\, \lambda^3) \text{ in } \mathcal{O}_{k_0}$,
where $\operatorname{Gal}(k/L)=\langle\tau\rangle$ and $\lambda=1-\zeta_3$.
\end{lemma}

\begin{proof}
The prime element $\lambda=1-\zeta_3$ divides $3$ in $k_0$:
$\lambda^2=1-2\zeta_3+\zeta_3^2=1-2\zeta_3-1-\zeta_3=-3\zeta_3$, because $1+\zeta_3+\zeta_3^2=0$.
Since $3=-\zeta_3^{-1}\lambda^2=-\zeta_3^2\lambda^2$, we have $9=3^2=(-\zeta_3^2\lambda^2)^2=\zeta_3\lambda^4$.
Consequently, $p\equiv 4$ or $7\,(\mathrm{mod}\,9)$ implies $\pi_1\pi_2=p\equiv 4$ or $7(\mathrm{mod}\, \lambda^3)$.
Then $\pi_1$ or $\pi_2\not\equiv 1(\mathrm{mod}\, \lambda^3)$, because otherwise
$\pi_1\equiv\pi_2\equiv 1(\mathrm{mod}\, \lambda^3)$,
$p=\pi_1\pi_2\equiv\pi_2\equiv 1(\mathrm{mod}\, \lambda^3)$,
and we obtain the contradiction $1\equiv 4$ or $7(\mathrm{mod}\, \lambda^3)$.
\end{proof}


\begin{lemma}
\label{pilamdaQ}
Suppose that $q$ is a prime number congruent to $2$ or $5$ mod $9$,
and thus $\pi=-q$ congruent to $4$ or $7$ mod $9$ is a prime element of $k_0$. \\
If $\pi\equiv 1(\mathrm{mod}\, 3\mathcal{O}_{k_0})$, then
$\pi\not\equiv 1(\mathrm{mod}\, \lambda^3) \text{ in } \mathcal{O}_{k_0}$,
where $\lambda=1-\zeta_3$.
\end{lemma}

\begin{proof}
Similarly as in the proof of Lemma \ref{pilamdaP},
the assumption $\pi\equiv 1(\mathrm{mod}\, \lambda^3)$ would yield the contradiction $1\equiv 4$ or $7(\mathrm{mod}\, \lambda^3)$.
\end{proof}


The invariant $q^{\ast}$, defined in the summary of our notation at the end of the Introduction,
is related to the \textit{norm index}
$(E_{k_0}:(E_{k_0}\cap N_{k/k_0}(k^\times)))$
by the relation
$$(E_{k_0}:(E_{k_0}\cap N_{k/k_0}(k^\times)))
=\frac{(E_{k_0}:E_{k_0}^3)}{((E_{k_0}\cap N_{k/k_0}(k^\times)):E_{k_0}^3)}
=3^{1-q^{\ast}}
\text{ i.e. }3^{q^{\ast}}=((E_{k_0}\cap N_{k/k_0}(k^\times)):E_{k_0}^3).$$

\begin{lemma}
\label{lem:NormIndex}
The invariant $q^{\ast}$ of a cubic Kummer extension $K=k_0(\sqrt[3]{\vartheta})$ of $k_0$
with radicand $\vartheta=\lambda^{e_{\lambda}}\zeta_3^{e_{\zeta}}\pi_1^{e_1}\cdots\pi_g^{e_g}$,
where $0\le e_{\lambda},e_{\zeta}\le 2$, $g\ge 0$,
$\pi_i\equiv 1\,(\mathrm{mod}\,3\mathcal{O}_{k_0})$ are prime elements of $k_0$, and $1\le e_i\le 2$, for $1\le i\le g$,
is given by
\begin{equation}
\label{eqn:NormIndex}
\begin{aligned}
q^{\ast}=1 & \Longleftrightarrow (\exists\,Z\in K)\,N_{K/k_0}(Z)=\zeta_3
             \Longleftrightarrow (\forall\,1\le i\le g)\,\pi_i\equiv 1\,(\mathrm{mod}\,\lambda^3) \text{ in } \mathcal{O}_{k_0}, \\
q^{\ast}=0 & \Longleftrightarrow (\forall\,Z\in K)\,N_{K/k_0}(Z)\neq\zeta_3
             \Longleftrightarrow (\exists\,1\le i\le g)\,\pi_i\equiv 4,7\,(\mathrm{mod}\,\lambda^3) \text{ in } \mathcal{O}_{k_0}.
\end{aligned}
\end{equation}
\end{lemma}

\begin{proof}
This lemma is due to Gerth \cite[\S\ 5, p. 92]{Ge1976}.
\end{proof}



Now, we start our proof of Theorem \ref{thm:Rank1}:
According to equation $(3.2)$ of \cite[p. 55]{Ge1975}, the radicand $d$ of $k=k_0(\sqrt[3]{d})$ can be written in the form
\begin{eqnarray}
\label{eq2}
d &=& 3^e. p_1^{e_1}......p_v^{e_v}p_{v+1}^{e_{v+1}}....p_w^{e_w}.q_1^{f_1}...q_I^{f_I}q_{I+1}^{f_{I+1}}...q_J^{f_J},
\end{eqnarray}
where $p_i$ and $q_i$ are positive rational prime numbers such that

\[
  \left\lbrace
  \begin{array}{l l l l l l l}
    p_i\equiv 1\,(\mathrm{mod}\,9),  & \quad \text{ for } \quad 1\leq i\leq v, \\
    p_i\equiv 4 \text{ or } 7\,(\mathrm{mod}\,9), & \quad \text{ for } \quad v+1\leq i\leq w,\\
    q_i\equiv -1\,(\mathrm{mod}\,9), & \quad \text{ for } \quad 1\leq i\leq I,\\
    q_i\equiv 2 \text{ or } 5  \,(\mathrm{mod}\,9), & \quad \text{ for } \quad I+1\leq i\leq J,\\
    e_i=1 \text{ or } 2,  & \quad \text{ for } \quad 1\leq i\leq w,\\
    f_i=1 \text{ or } 2, & \quad \text{ for } \quad 1\leq i\leq J,\\
    e=0, 1 \text{ or } 2.
  \end{array}
  \right.
\]

Let $C_{k,3}^{(\sigma)}=\lbrace\mathcal{A}\in C_{k,3}\mid\mathcal{A}^{\sigma}=\mathcal{A}\rbrace$
be the $3$-group of ambiguous ideal classes of $k/k_0$,
where $\operatorname{Gal}(k/k_0)=\langle\sigma\rangle$ and $C_{k,3}$ is the $3$-class group of $k$.
According to \cite[\S\ 5, p. 92]{Ge1976}, the $3$-rank of $C_{k,3}^{(\sigma)}$ is specified as follows:
$$\operatorname{rank}\,(C_{k,3}^{(\sigma)})=t-2+q^{\ast},$$
where $t$ is the number of prime ideals of $k_{0}$ ramified in $k$,
and \(q^{\ast}\) is the invariant in Lemma \ref{lem:NormIndex}.

If we assume \(\operatorname{rank}(C_{k,3}^{(\sigma)})=1\), $t+q^{\ast}=3$, according to the main goal of this paper,
then we have the following cases, according to Lemma $3.1$ of \cite[p. 55]{Ge1975}:
   \begin{itemize}
    \item Case 1: \(2w+J=1\),
    \item Case 2: \(2w+J=2\),
    \item Case 3: \(2w+J=3\),
   \end{itemize}
where $w$ and $J$ are the positive integers defined in equation (\ref{eq2}).
We shall successively treat these cases, distinguishing subcases in dependence on $w$ and $J$,
and necessarily also obtaining some side results concerning
the situation \(\operatorname{rank}(C_{k,3}^{(\sigma)})=0\), $t+q^{\ast}=2$,
studied by Honda \cite{Ho} and refined in our section \ref{ss:Rank0},
and the situation \(\operatorname{rank}(C_{k,3}^{(\sigma)})=2\), $t+q^{\ast}=4$,
which will be investigated in our forthcoming paper \cite{AMI}.




\subsection{Radicands divisible by a single non-split prime}
\label{ss:SingleNonSplitPrime}

\noindent
In the Case 1, \(2w+J=1\),
we must have \(w=0\) and \(J=1\).
Then \(d=3^{e}q^{f_1}\), where $q$ is a prime number such that \(q\equiv -1(\mathrm{mod}\, 3)\), $e\in\{0,1,2\}$ and $f_1\in\{1,2\}$.

   \subsubsection{Species 2}
   \noindent
   If \(d\equiv\pm 1\,(\mathrm{mod}\,9)\),
   then $e=0$ and $q\equiv -1\,(\mathrm{mod}\,9)$.
        Since $q\equiv -1\,(\mathrm{mod}\,3)$, $q$ remains inert in $k_0$.
         As $d\equiv\pm 1(\mathrm{mod}\, 9)$, $3$ is decomposed in $L$
         and $\lambda$ is not ramified in $k/k_0$. We get $t=1$.
          Put $x=\pi^{f_1}$, where $\pi=-q$ is a prime element of $k_0$,
          then $k=k_0(\sqrt[3]{x})$, because \(d=q^{f_1}\).
          In addition, since $q\equiv -1\,(\mathrm{mod}\,9)$, then by Lemma \ref{pilamdaQ}, $\pi$ is congruent to $1(\mathrm{mod}\, \lambda^3)$,
          and according to Lemma \ref{lem:NormIndex}, $\zeta_3$ is
           norm of an element of $k\setminus\lbrace 0\rbrace$ and $q^{\ast}=1$.
           Thus $\operatorname{rank}\,(C_{k,3}^{(\sigma)})=0$, which is item $(2)$ of Theorem \ref{thm:Honda}.

   \subsubsection{Species 1}
   \noindent
   If \(d\not\equiv\pm 1\,(\mathrm{mod}\,9)\), then
   \begin{itemize}
    \item
    either $q\equiv -1\,(\mathrm{mod}\,9)$, then $d=3^{e}q^{f_1}\equiv\pm 3^{e}\ (\mathrm{mod}\, 9)$, so $e\neq 0$, which is item $(1)$ of Theorem \ref{thm:Ismaili2},
    \item
    or $q\equiv 2$ or $5\,(\mathrm{mod}\,9)$, whence $q$ remains inert in $k_0$.
    Moreover, since $d\not\equiv\pm 1\ (\mathrm{mod}\, 9)$, $3$ is totally ramified in $L$,
  so $\lambda$ is ramified in $k/k_0$.
  Thus, $t=2.$
  We have $3=-\zeta_3^2\lambda^2$,
  and $k=k_0(\sqrt[3]{x})$ with $x=\zeta_3^2\lambda^2\pi^{f_1}$,
  where $\pi=-q$ is a prime element of $k_0$.
  Since $q\equiv 2$ or $5\,(\mathrm{mod}\,9)$, so by Lemma \ref{pilamdaQ},
   the prime $\pi$ is not congruent to $1(\mathrm{mod}\, \lambda^3)$,
    and according to Lemma \ref{lem:NormIndex}, $\zeta_3$ is not
     norm of an element of $k\setminus\lbrace 0\rbrace$. So $q^{\ast}=0$, and we conclude that
      $\operatorname{rank}\,(C_{k,3}^{(\sigma)})=0$, which gives items $(3),(4)$ of Theorem \ref{thm:Honda}.
       \end{itemize}
       
\noindent
We conclude that \(\operatorname{rank}(C_{k,3}^{(\sigma)})=1\)
at most for \(d=3^{e}q^{f_1}\not\equiv\pm 1(\mathrm{mod}\, 9)\), with \(q\equiv -1\,(\mathrm{mod}\,9)\),
and $e,f_1\in\lbrace 1,2\rbrace$.
This is item $(1)$ of Theorem \ref{thm:Ismaili2}. \\


 

In the Case 2, \(2w+J=2\),
we either have $w=1$, $J=0$, treated in \S\ \ref{ss:SingleSplitPrime},
or $w=0$, $J=2$, treated in \S\ \ref{ss:TwoNonSplitPrimes}.
   
\subsection{Radicands divisible by a single split prime}
\label{ss:SingleSplitPrime}
   \noindent
   If $w=1$ and $J=0$, then \(d=3^{e}p^{e_1}\), where $p$ is a prime number such that
   \(p \equiv 1(\mathrm{mod}\, 3)\), $e\in\{0,1,2\}$ and $e_1\in \{1,2\}$.

   \subsubsection{Species 2}
   \noindent
   If \(d\equiv\pm 1\,(\mathrm{mod}\,9)\), then \(p \equiv 1 \pmod  9\) and $e=0$.
   So the radicand $d$ can be written in the form  $d=p^{e_1}$, where $p\equiv 1 \,(\mathrm{mod}\,9)$ and $e_1\in\{1,2\}$.
   This is item $(1)$ of Theorem \ref{thm:Ismaili1}.
   
   \subsubsection{Species 1}
   \noindent
   If \(d \not \equiv \pm 1 \pmod  9\), we have two possibilities: $e=0$ and $p \not \equiv 1 \ (\mathrm{mod}\, \ 9) $,  or $e=1$ or $2$.
   Then the possible forms of $d$ are:
   \[ d = \left\{
  \begin{array}{l l l}
    p^{e_1} & \quad \text{ with } p \equiv 4 \text{ or } 7\,(\mathrm{mod}\,9),\\
   3^{e}p^{e_1} \not \equiv \pm 1 \,(\mathrm{mod}\,9) & \quad \text{ with } p \equiv 1 \pmod  3,\\
  \end{array} \right.\]
where $e,e_1\in\lbrace 1,2\rbrace$.  
The former is item $(2)$ of Theorem \ref{thm:Ismaili1}, the latter must be examined further. \\

\noindent
Now, assume that $d=3^{e}p^{e_1} \not \equiv \pm1 \pmod  9$, with $
p \equiv 1  \pmod  3$.  As
$\mathbb{Q}(\sqrt[3]{ab^2})=\mathbb{Q}(\sqrt[3]{a^2b})$ for square-free coprime $a,b \in \mathbb{N}$,
we can choose $e_1=1$, i.e. $d=3^{e}p$ with $e\in
\lbrace 1,2 \rbrace$. The fact that  $ p \equiv 1  \pmod  3$ implies
that
  $p=\pi_1 \pi_2$, where $\pi_1$ and $\pi_2$ are prime elements of $k_0$
  such that $\pi_1^{\tau}=\pi_2$ and $\pi_1 \equiv \pi_2 \equiv 1 ~(\mathrm{mod}\,~3\mathcal{O}_{k_0})$,
   furthermore the prime $p$ is totally ramified in $L$. Thus $\pi_1$ and $\pi_2$
   are totally ramified in $k$ and we have $\pi_1\mathcal{O}_{k}=\mathcal{P}_1^3$
and $\pi_2 \mathcal{O}_{k}=\mathcal{P}_2^3$ where $\mathcal{P}_1, \mathcal{P}_2$
are two prime ideals of $k$. Since $d \not \equiv \pm 1 ~(\mathrm{mod}\,~9) $,
 $3$ is totally ramified in $L$, so $\lambda$ is ramified in $k/k_0$.
  Hence, the number of prime ideals of $k_0$ which are ramified in $k/k_0$ is $t=3.$\\
Since $3=-\zeta_3^2\lambda^2$, we have
$k=k_0(\sqrt[3]{x})$, where $x=\zeta_3^2\lambda^2
\pi_1 \pi_2$. If $
p \equiv 1 \,(\mathrm{mod}\,9)$, the primes $\pi_1$ and $\pi_2$ are congruent to
$1(\mathrm{mod}\, \lambda^3)$, and according to Lemma \ref{lem:NormIndex},
  we have $q^{\ast}=1$, i.e. $\zeta_3$ is the norm of an element of $k\setminus\lbrace 0\rbrace$.
  Thus $\operatorname{rank}\,(C_{k,3}^{(\sigma)})=2$, which we remember for \cite{AMI}.
  For $\operatorname{rank}\,(C_{k,3}^{(\sigma)})=1$ we necessarily have $p \equiv 4$ or $7\,(\mathrm{mod}\,9)$,
  which yields item $(3)$ of Theorem \ref{thm:Ismaili1}. 
 
 Hence, the forms of the integer $d$ with \(\operatorname{rank}(C_{k,3}^{(\sigma)})=1\) are:
    \[ d = \left\{
  \begin{array}{l l l}
    p^{e_1} & \quad \text{ with } p\equiv 1 \,(\mathrm{mod}\,3), \\
   3^{e}p^{e_1} \not \equiv \pm1 \pmod  9 & \quad \text{ with } p\equiv 4 \text{ or } 7 \,(\mathrm{mod}\,9). \\

  \end{array} \right.\]
where $e,e_1\in\lbrace 1,2\rbrace$. These are the items $(1),(2),(3)$ of Theorem \ref{thm:Ismaili1}.



\subsection{Radicands divisible by two non-split primes}
\label{ss:TwoNonSplitPrimes}
\noindent
If $w=0$ and $J=2$, then
   \(d=3^{e}q_1^{f_1}q_2^{f_2}\), with
    \(q_1\equiv q_2\equiv 2 (\mathrm{mod}\, 3)\), $e\in\{0,1,2\}$ and $f_1,f_2\in\{1,2\}$.
    
   \subsubsection{Species 2}
   \noindent
    If \(d\equiv\pm 1\,(\mathrm{mod}\,9)\), we have
    \(d=3^{e}q_1^{f_1}q_2^{f_2}\equiv\pm 3^{e}\,\text{ or }\, \pm 3^{e}\times 2\,\text{ or }\,\pm 3^{e}\times 5\,(\mathrm{mod}\,9)\)
   whence $e=0$.
   If \(q_1,q_2\equiv 2,5\,(\mathrm{mod}\,9)\), then \(t=2\) and \(q^{\ast}=0\),
   which is item $(5)$ of Theorem \ref{thm:Honda} with $\operatorname{rank}\,(C_{k,3}^{(\sigma)})=0$.
   Consequently, we must have \(q_1\equiv q_2\equiv -1\,(\mathrm{mod}\,9)\) and \(q^{\ast}=1\) for rank $1$.
   So the radicand $d$ has the form $d=q_1^{f_1}q_2^{f_2}$
   with $q_1\equiv q_2\equiv -1\,(\mathrm{mod}\,9)$ and $f_1,f_2\in\{1,2\}$, i.e. item $(2)$ of Theorem \ref{thm:Ismaili2}.

   \subsubsection{Species 1}
   \noindent
     If \(d\not\equiv\pm 1\,(\mathrm{mod}\,9)\), we either have
     $e \neq 0$ or ($e=0$ and there exist $i \in \lbrace 1,2 \rbrace$ such that $q_{i}\not\equiv -1\,(\mathrm{mod}\,9)$).
     The latter case yields items $(3),(4)$ of Theorem \ref{thm:Ismaili2}. \\
     Assume that $e \neq 0$, then \(d=3^{e}q_1^{f_1}q_2^{f_2}\)
     with \(q_1\equiv q_2\equiv 2\ (\mathrm{mod}\, 3)\), $e,f_1,$ and $f_2\in\{1,2\}$.
     For each $i\in\lbrace 1,2\rbrace$, $q_{i}$ is inert in $k_0$,
     and $q_{i}$ is ramified in $L$. As $d \not \equiv \pm1 \pmod  9$, $3$
     is ramified in $L$, so $\lambda$ is ramified in $k/k_0$.
     Thus $t=3$. The fact that $3=-\zeta_3^2\lambda^2$ implies that $k=k_0(\sqrt[3]{x})$
     with $x=\zeta_3^2\lambda^2 \pi_1^{f_1}\pi_2^{f_2}$, where $-q_{i}=\pi_{i}$
     is a prime element of $k_0$, for  $i \in \lbrace 1,2 \rbrace$.\\
     If \(q_1\equiv q_2\equiv -1\,(\mathrm{mod}\,9)\),
      then  all primes $\pi_1$, $\pi_2$ are congruent to $ 1 ~(\mathrm{mod}\,~\lambda^3)$,
      and according to Lemma \ref{lem:NormIndex}, $\zeta_3$
      is  norm of an element of $k\setminus\lbrace 0\rbrace$ and $q^{\ast}=1$.
      We conclude that $\operatorname{rank}\,(C_{k,3}^{(\sigma)})=2$, which we keep in mind for \cite{AMI}.
      Thus, for rank $1$, there exists $i\in\lbrace 1,2\rbrace$ such that $q_{i}\not\equiv -1\,(\mathrm{mod}\,9)$,
      which is item $(5)$ of Theorem \ref{thm:Ismaili2}.
      
     Hence, the form of $d$ for $\operatorname{rank}\,(C_{k,3}^{(\sigma)})=1$ is \(d=3^{e}q_1^{f_1}q_2^{f_2}\not\equiv\pm 1\,(\mathrm{mod}\,9)\),
     with $e\in\lbrace 0,1,2\rbrace$, $f_1,f_2\in\lbrace 1,2\rbrace$, and there exists $i\in\lbrace 1,2\rbrace$
     such that $q_{i}\not\equiv -1\,(\mathrm{mod}\,9)$. These are the items $(3),(4),(5)$ of Theorem \ref{thm:Ismaili2}.\\


 

In the Case 3, \(2w+J=3\),
we either have \(w=1\), \(J=1\), treated in \S\ \ref{ss:SplitAndNonSplitPrime},
or $w=0$, $J=3$, treated in \S\ \ref{ss:ThreeNonSplitPrimes}.
   
\subsection{Radicands divisible by a split prime and a non-split prime}
\label{ss:SplitAndNonSplitPrime}
\noindent
   If \(w=1\) and \(J=1\), then
   \(d=3^{e}p^{e_1}q^{f_1}\), where $p$ and $q$ are prime numbers such that \(p \equiv 1 (\mathrm{mod}\, 3) \) and \(q\equiv  2 (\mathrm{mod}\, 3) \),
   $e\in\{0,1,2\}$ and $e_1,f_1\in\{1,2\}$.
   
   \subsubsection{Species 2}
   \noindent
   If \(d \equiv \pm1 \pmod  9 \), then necessarily $e=0$, and we distinguish the following cases:
   \begin{itemize}

   \item[1)] If \(p \equiv 4\, \ \text{or} \ \,7 \pmod  9\) and \(q \equiv -1 (\mathrm{mod}\, 9)\): \\
    then we have
    \(d=p^{e_1} q^{f_1} \equiv \pm 4 \, \text{ or }\, \pm 7 (\mathrm{mod}\, \ 9)\) and  \(d \not \equiv \pm 1 \pmod  9 \) cannot be of species 2.

  \item[2)] If \(p \equiv -q \equiv 4\, \ \text{or}\, \ 7 \pmod  9\):\\
    we have \(p^{e_1} \equiv \pm 4\, \text{or}\, \pm 7 (\mathrm{mod}\, 9)\) and \(q^{f_1} \equiv \pm 2\, \text{or}\, \pm 5 \pmod  9 \), so
   \(p^{e_1}q^{f_1}\equiv \pm 1\, \text{ or }\, \pm 5 \, \text{ or } \, \pm 2 \ (\mathrm{mod}\, 9)\),
  and the radicand $d= p^{e_1}q^{f_1}\equiv \pm1 \pmod  9$ belongs to item $(4)$ of Theorem \ref{thm:Ismaili1}.

 \item[3)] If $p \equiv -q \equiv 1  \,(\mathrm{mod}\,9)$ : \\
 then $d= p^{e_1}q^{f_1}$.
 Since $d \equiv \pm1 \pmod  9$, $3$ is not ramified in $L$, so $\lambda$ is not ramified in $k/k_0$.
 Since  $ p \equiv 1  \pmod  3$, $p=\pi_1 \pi_2$, where $\pi_1$ and $\pi_2$ are prime elements of $k_0$ such that 
 $\pi_1^{\tau}=\pi_2$ and $\pi_1 \equiv \pi_2 \equiv 1 ~(\mathrm{mod}\,~3\mathcal{O}_{k_0})$.
 The prime $p$ is totally ramified in $L$, so $\pi_1$ and $\pi_2$ are totally ramified in $k$.
 Since  $ q \equiv -1  \pmod  3$, $q$ remains inert in $k_0$.
 Thus the prime ideals ramified in $k/k_0$ are those generated by $\pi_1, \pi_2$ and $q$, whence $t=3$. \\
 The fact that $ p \equiv -q \equiv 1  \,(\mathrm{mod}\,9)$ implies that $\pi_1 \equiv \pi_2 \equiv \pi \equiv 1 ~(\mathrm{mod}\,~\lambda^3)$,
 and $-q=\pi$ is a prime element of $k_0$.
 Put $x=\pi_1^{e_1} \pi_2^{e_1} \pi^{f_1}$, then $k=k_0(\sqrt[3]{x})$. 
 Since all primes $\pi$, $\pi_1$ and $\pi_2$ are congruent to $ 1 ~(\mathrm{mod}\,~\lambda^3)$,
 $\zeta_3$ is the norm of an element of $k\setminus\lbrace 0\rbrace$ and $q^{\ast}=1$,
 according to Lemma \ref{lem:NormIndex}.
 We conclude that $\operatorname{rank}\,(C_{k,3}^{(\sigma)})=2$, which we remember for \cite{AMI}.

   \item[4)] If \(p \equiv 1 \pmod  9\) and \(q \equiv 2\, \ \text{or}\, 5 \ (\mathrm{mod}\, 9)\): \\
   then we have
   \(d=p^{e_1} q^{f_1}\equiv \pm 2 \, \text{ or } \, \pm 5 (\mathrm{mod}\, 9)\) and \(d \not \equiv \pm 1 \,(\mathrm{mod}\,9)\) cannot be of species 2.
\end{itemize}

   \subsubsection{Species 1}
   \noindent
 Let \(d \not \equiv \pm1 \pmod  9\). \\
  On the one hand, \(\operatorname{rank}\,(C_{k,3}^{(\sigma)})=t-2+q^*\),
  and our desired fact \(\operatorname{rank}\,(C_{k,3}^{(\sigma)})=1\) implies $t \leq 3.$\\
   On the other hand, we have \(d=3^{e}p^{e_1}q^{f_1}\), with \(p \equiv 1 \pmod
   3 \) and \(q\equiv  2 \,(\mathrm{mod}\,3) \).
   We calculate the number $t$ of prime ideals which ramify in $k/k_0$.
   Since $p\equiv 1 \ (\mathrm{mod}\, \ 3$), $p=\pi_1\pi_2$, where $\pi_1$ and
$\pi_2$ are two prime elements of $k_{0}$ such that
$\pi_2=\pi_1^{\tau}$ and $\pi_1 \equiv \pi_2 \equiv 1 ~(\mathrm{mod}\,~3\mathcal{O}_{k_0})$.
The prime $p$ is ramified in $L$, whence $\pi_1$ and $ \pi_2$ are ramified in $k$.
$q$ remains inert in $k_0$, and $q$ is ramified in $L$. Since \(d \not \equiv \pm1 \pmod  9\),
$3$ is ramified in $L$,
and we have $3\mathcal{O}_{k_0}=(\lambda)^2$ where $\lambda=1-\zeta_3.$
Hence, $t=4$ and \(\operatorname{rank}\,(C_{k,3}^{(\sigma)})\ge 2\), which we note for examination in \cite{AMI}.

\subsection{Radicands divisible by three non-split primes}
\label{ss:ThreeNonSplitPrimes}
\noindent
If $w=0$ and $J=3$, then
   \(d=3^{e}q_1^{f_1}q_2^{f_2}q_3^{f_3}\), where $q_{i}$ is a prime number such that
    \(q_{i}\equiv  2  \pmod  3 \), $e\in   \{0,1,2\}$ and $f_{i} \in   \{1,2\}$
    for each $i\in   \{1,2,3\}$. 
    
   \subsubsection{Species 2}
   \noindent
   If \(d \equiv \pm1 \pmod  9 \), then generally $e=0$, and we distinguish the following cases:
    \begin{itemize}

   \item[1)] If \(q_1 \equiv 2\, \ \text{or} \ \,5 \pmod  9\) and \(q_2 \equiv q_3 \equiv -1 \ (\mathrm{mod}\, \  9)\): \\
   we get \(d=q_1^{f_1}q_2^{f_2}q_3^{f_3} \equiv \pm 2\, \text{ or }\, \pm 5 (\mathrm{mod}\, \  9)\),
   and \(d \not \equiv \pm 1 \pmod  9\) cannot be of species 2.

  \item[2)] If \(q_1 \equiv q_2 \equiv 2\, \ \text{or} \ \,5 \pmod  9\) and \( q_3 \equiv -1 \pmod  9\): \\
     we have \(q_1^{f_1} \equiv q_2^{f_2} \equiv \pm 2\, \text{ or }\, \pm 5 (\mathrm{mod}\,  9)\), so
   \(q_1^{f_1}q_2^{f_2}\equiv \pm 1\, \text{or}\, \pm 2 \, \text{ or } \, \pm 5 (\mathrm{mod}\,  9)\),
    and the possible form of the radicand $d$ is $d= q_1^{f_1}q_2^{f_2}q_3^{f_3} \equiv \pm1 \pmod  9$,
   i.e. item $(7)$ of Theorem \ref{thm:Ismaili2}.
 
 \item[3)] If \(q_1 \equiv q_2 \equiv q_3 \equiv -1  \pmod  9\): \\
 then we have $d=q_1^{f_1}q_2^{f_2}q_3^{f_3} \equiv  \pm 1 \pmod  9\).
   For each $i \in \{1,2,3\}$, $q_{i}$ remains inert in $k_0$, and $q_{i}$ is ramified in $L$.
   As $d \equiv \pm1 \pmod  9$, $3$ is decomposed in $L$, so $\lambda$ is not ramified in $k/k_0$.
   Thus $t=3$.
   The fact that $q_{i} \equiv -1  \,(\mathrm{mod}\,9)$ implies that $\pi_{i}  \equiv 1 ~(\mathrm{mod}\,~3\mathcal{O}_{k_0})$,
   where $-q_{i}=\pi_{i}$ is a prime element of $k_0$.
   Put $x=\pi_1^{f_1} \pi_2^{f_2}\pi_3^{f_3}$, then
 $k=k_0(\sqrt[3]{x})$.  Since all primes $\pi_1$, $\pi_2$ and $\pi_3$ are congruent to $ 1 ~(\mathrm{mod}\,~\lambda^3)$,
 then according to Lemma \ref{lem:NormIndex},
 $\zeta_3$ is a norm of an element of $k\setminus\lbrace 0\rbrace$ and $q^{\ast}=1$.
 We conclude that $\operatorname{rank}\,(C_{k,3}^{(\sigma)})=2$, which we note for \cite{AMI}.

   \item[4)] If \(q_1 \equiv q_2 \equiv q_3 \equiv 2 \, \text{or} \,  5  \pmod  9\): \\
   then  we have \(d=q_1^{f_1}q_2^{f_2}q_3^{f_3}\equiv \pm 1\, \text{ or }\, \pm 2 \,
   \text{or} \, \pm 5  \pmod  9\), 
   and the possible form of the radicand $d$ is $d= q_1^{f_1}q_2^{f_2}q_3^{f_3}\equiv \pm1 \pmod  9$,
   which is item $(6)$ of Theorem \ref{thm:Ismaili2}.
   
\end{itemize}

   \subsubsection{Species 1}
   \noindent
   If \(d \not \equiv \pm 1 \pmod  9\), we should have $t \leq 3$ for \(\operatorname{rank}\,(C_{k,3}^{(\sigma)})=1\),
   but we reason similar as in the case $w=1$, $J=1$, and we get $t=4$, \(\operatorname{rank}\,(C_{k,3}^{(\sigma)})\ge 2\),
   which we remember for examination in \cite{AMI}. \\



\noindent
Finally, all forms of the integer $d$ for which $\operatorname{rank}\,(C_{k,3}^{(\sigma)})=1$ can be summarized as:

\[ d = \left\lbrace
  \begin{array}{l l l l l l l l l}
   p_1^{e_1} & \quad\text{where \ $ p_1 \equiv 1\,(\mathrm{mod}\,3)$,}\\
   3^{e}p_1^{e_1} & \quad\text{where \ $p_1\equiv 4$ or $7\,(\mathrm{mod}\,9)$,}\\
   p_1^{e_1}q_1^{f_1} \equiv\pm 1\,(\mathrm{mod}\,9) & \quad\text{where \ $p_1,-q_1\equiv 4$ or $7\,(\mathrm{mod}\,9)$,}\\
   3^{e}q_1^{f_1} & \quad\text{where \ $q_1\equiv -1\,(\mathrm{mod}\,9)$,}\\
   q_1^{f_1}q_2^{f_2} & \quad\text{where \ $q_1\equiv q_2\equiv -1\,(\mathrm{mod}\,9)$,}\\
   q_1^{f_1}q_2^{f_2} \not\equiv\pm 1\,(\mathrm{mod}\,9) & \quad\text{where \ $q_{i}\not\equiv -1\,(\mathrm{mod}\,9)$ for some $i\in\lbrace 1,2\rbrace$,}\\
   3^{e}q_1^{f_1}q_2^{f_2}    & \quad\text{where \ $q_{i}\not\equiv -1\,(\mathrm{mod}\,9)$ for some $i\in\lbrace 1,2\rbrace$,}\\
   q_1^{f_1}q_2^{f_2}q_3^{f_3} \equiv\pm 1\,(\mathrm{mod}\,9) & \quad\text{where \ $q_1,q_2\equiv 2$ or $5\,(\mathrm{mod}\,9)$ and $q_3\equiv -1\,(\mathrm{mod}\,9)$,}\\
   q_1^{f_1}q_2^{f_2}q_3^{f_3} \equiv\pm 1\,(\mathrm{mod}\,9) & \quad\text{where \ $q_1,q_2,q_3\equiv 2$ or $5\,(\mathrm{mod}\,9)$,}\\
  \end{array} \right.\]
and where $e, e_1, f_1, f_2$ and $f_3$ are positive integers equal to $1$ or $2$.



\subsection{Proof of rank one for each case}
\label{ss:ProofRank1}

\noindent
Conversely, we prove $\operatorname{rank}\,(C_{k,3}^{(\sigma)})=1$, if the radicand $d$ takes one of the forms in Theorem \ref{thm:Rank1}.

\subsubsection{Case where $d= p^{e_1}$, with $ p \equiv 1  \pmod 9 $}
   The fact that  $ d \equiv 1  \pmod 9 $ implies that the pure cubic field $L$ is of Dedekind's species 2 (see \cite[\S 13]{BC1971}), so $3$ decomposes in $L$, and then $\lambda$ decomposes in $k/k_0$. Moreover, as
  $ p \equiv 1  \pmod  3$, we have  $p=\pi_1 \pi_2$, where $\pi_1$ and $ \pi_2$ are primes of $k_0$ such that $\pi_1^{\tau}=\pi_2$ and $\pi_1 \equiv \pi_2 \equiv 1 ~(\bmod~3\mathcal{O}_{k_0})$. It follows that $p$ is totally ramified in $L$, and therefore $\pi_1$ and $\pi_2$ are totally ramified in $k$. Thus, $t=2$. Taking into account that $ p \equiv 1  \pmod 9$, the Hilbert cubic symbol:
  $$\left(\frac{\zeta_3,p}{\pi_1} \right)_3 = \left(\frac{\zeta_3,p}{\pi_2} \right)_3=1.$$
  It follows by Lemma \ref{lem:NormIndex} that $\zeta_3$ is a norm of an element of $k\setminus\lbrace 0\rbrace$, so $q^{\ast}=1$. Hence $\operatorname{rank}\,(C_{k,3}^{(\sigma)})=1.$


\subsubsection{Case where $d= p^{e_1}$, with $ p \equiv 4$ or $7  \pmod 9$}
Since $d \not \equiv \pm 1 ~(\bmod~9) $, $3$ is totally ramified in $L$, and then $\lambda$ is ramified in $k/k_0$. In the same manner of case (1), the congruence $ p \equiv 1   \pmod  3 $ implies that
$\pi_1$ and $\pi_2$ are totally ramified in $k$, where $\pi_i$ for $i=1 $ or $2$ are defined as above. Then $t=3.$   As $ p \equiv 4$ or $7  \pmod 9$, the Hilbert cubic symbol
  $$\left(\frac{\zeta_3,p}{\pi_i} \right)_3 \neq 1,$$
  for $i=1$ or $2$. It follows by Lemma \ref{lem:NormIndex} that $\zeta_3$ is not  a norm of an element of $k\setminus\lbrace 0\rbrace$ and $q^{\ast}=0$. Thus $\operatorname{rank}\,(C_{k,3}^{(\sigma)})=1.$

\subsubsection{Case where $d=3^{e}p^{e_1} \not \equiv \pm1 \pmod  9$, with $
p \equiv 4$ or $7 \pmod 9$}

 Taking into account that
$\mathbb{Q}(\sqrt[3]{ab^2})=\mathbb{Q}(\sqrt[3]{a^2b})$ for square-free coprime $ a,
b \in \mathbb{N}$, we can choose $e_1=1$, i.e. $d=3^{e}p$ with $e\in
\lbrace 1,2 \rbrace$. The fact that  $ p \equiv 1  \pmod  3$ implies
that
  $p=\pi_1 \pi_2$, where $\pi_1$ and $\pi_2$ are primes of $k_0$
  such that $\pi_1^{\tau}=\pi_2$ and $\pi_1 \equiv \pi_2 \equiv 1 ~(\bmod~3\mathcal{O}_{k_0})$,
   then the prime $p$ is totally ramified in $L$, and so the two primes $\pi_1$ and $\pi_2$
   are totally ramified in $k$. It follows by $d \not \equiv \pm 1 ~(\bmod~9) $ that $3$ is totally ramified in $L$, then $\lambda$ is ramified in $k/k_0$. Hence $t=3.$ \\
Further, by $3=-\zeta_3^2\lambda^2$ we obtain
$k=k_0(\sqrt[3]{x})$, where $x=\zeta_3^2\lambda^2
\pi_1 \pi_2$. Having as a fact $ p \equiv  4$ or $7 \pmod 9$ and by Lemma \ref{pilamdaP} entails that the primes $\pi_1$ and $\pi_2$ are not congruent to
$ 1 ~(\bmod~\lambda^3)$, then according to Lemma \ref{lem:NormIndex},
we obtain $q^{\ast}=0$ and $\zeta_3$ is a not norm of an element of $k\setminus\lbrace 0\rbrace$. Therefore, $\operatorname{rank}\,(C_{k,3}^{(\sigma)})=1$.

\subsubsection{Case where $d=p^{e_1}q^{f_1} \equiv \pm1 \pmod 9$, with $ p \equiv -q \equiv 4$ or $7  \pmod 9$}
It follows by $d \equiv \pm1 \pmod  9$ that $3$ is not ramified in the pure cubic field $L$, then $\lambda$ is also not ramified in $k/k_0$.
 The congruence  $ p \equiv 1  \pmod  3$ implies that $p=\pi_1 \pi_2$ with $\pi_1^{\tau}=\pi_2$ and $\pi_1 \equiv \pi_2 \equiv 1 \pmod{3\mathcal{O}_{k_0}}$, then $\pi_1$ and $\pi_2$ are totally ramified in $k$. As  $ q \equiv -1  \pmod  3$,  $q$ is inert in $k_0$. Therefore, the primes ramified in $k/k_0$ are $\pi_1, \pi_2$ and $q$. \\   
 Afterwards, let be $x=\pi_1^{e_1} \pi_2^{e_1} \pi^{f_1}$, where $-q=\pi$ is a prime number of $k_0$, then
 $k=k_0(\sqrt[3]{x})$.  It is clear that the congruence $ q \equiv 2$ or $5  \pmod 9$ implies that $\pi$ is not congruent to $ 1 \pmod{\lambda^3}$, then from Lemma \ref{lem:NormIndex} we get $\zeta_3$ is not a norm of an element of $k\setminus\lbrace 0\rbrace$, and so $q^{\ast}=0$. Consequently, $\operatorname{rank}\,(C_{k,3}^{(\sigma)})=1$.

\subsubsection{Case where $d=3^{e}q^{f_1} \not \equiv \pm1 \pmod  9$, with  $ q \equiv -1  \pmod 9 $ }
 
   The congruence $d \not \equiv \pm 1 \pmod 9$ entails that $3$ is totally ramified in $L$, then $\lambda$ is ramified in $k/k_0$.  The fact that
   $ q \equiv -1   \pmod  3$ implies that $q$ is inert in $k_0$. Then, $t=2.$ \\
   Taking into account that $3=-\zeta_3^2\lambda^2$, we obtain $k=k_0(\sqrt[3]{x})$ with $x=\zeta_3^2\lambda^2 \pi^{f_1}$, where $-q=\pi$ is a prime number of $k_0$. As $ q \equiv -1  \pmod 9$, the prime $\pi$ is congruent to $ 1 ~(\bmod~\lambda^3)$, then by Lemma \ref{lem:NormIndex}, $\zeta_3$ is  a norm of an element of $k\setminus\lbrace 0\rbrace$ and $q^{\ast}=1$. Hence, $\operatorname{rank}\,(C_{k,3}^{(\sigma)})=1.$
   
\subsubsection{Case where $d=q_1^{f_1}q_2^{f_2}$, with  $ q_1 \equiv q_2 \equiv -1 \pmod  9$ }
  For each $i \in   \{1,2\}$, the prime $q_{i}$ is inert in $k_0$ because $ q_{i} \equiv 2 (\bmod 3)$, however $q_{i}$ is ramified in $L$.  The congruence $d \equiv \pm1 \pmod  9$ entails that $3$ decomposes in $L$ and then $\lambda$ is not ramified in $k/k_0$.  So $t=2$. The fact that $q_{i} \equiv -1  \pmod 9$ implies that $\pi_{i}  \equiv 1 ~(\bmod~\lambda^3)$, where $-q_{i}=\pi_{i}$ is a prime number of $k_0$. \\
    Afterwards, put $x=\pi_1^{f_1} \pi_2^{f_2}$, then
 $k=k_0(\sqrt[3]{x})$,
 and from Lemma \ref{lem:NormIndex} we see that $\zeta_3$ is
 a norm of an element of $k\setminus\lbrace 0\rbrace$, so $q^{\ast}=1$.
   We conclude that $\operatorname{rank}\,(C_{k,3}^{(\sigma)})=1$.

   \subsubsection{Case where $d=q_1^{f_1}q_2^{f_2}\not \equiv \pm1 \pmod  9$, and $ \exists i \in \lbrace 1,2 \rbrace \ | \ q_{i} \not \equiv -1  \pmod 9$ 
   }
  Having as a fact $ q_{i} \equiv 2 \pmod 3$,  for each $i \in   \{1,2\}$,  entails that $q_{i}$ is inert in $k_0$. Further, $q_{i}$ is ramified in $L$. The congruence $d \not \equiv  \pm1 \pmod  9$ implies that $3$ is ramified in $L$, and therefore $\lambda$ is ramified in $k/k_0$.  Then we obtain $t=3$. From hypothesis, there exist $i \in   \{1,2\}$ such that  $q_{i} \not \equiv -1  \pmod 9$, then $\pi_{i}  \not \equiv 1 ~(\bmod~\lambda^3)$ where $-q_{i}=\pi_{i}$ is a prime number of $k_0$. \\ 
   Afterwards, put $x=\pi_1^{f_1} \pi_2^{f_2}$, then we get
 $k=k_0(\sqrt[3]{x})$.
According to Lemma \ref{lem:NormIndex} we have $\zeta_3$ is not
  norm of an element of $k\setminus\lbrace 0\rbrace$ and $q^{\ast}=0$.
   Then $\operatorname{rank}\,(C_{k,3}^{(\sigma)})=1$.

 \subsubsection{Case where $d=3^{e}q_1^{f_1}q_2^{f_2}\not \equiv \pm1 \pmod  9 $, $e \in \lbrace 0,1,2 \rbrace$ and $ \exists i \in \lbrace 1,2 \rbrace \ | \ q_{i} \not \equiv -1  \pmod 9$ } 
 For each $i \in   \{1,2\}$, $q_{i}$ is ramified in $L$, and $\lambda$ is  ramified in $k/k_0$ because $d \not \equiv \pm1 \pmod  9$, so $t=3$. We reason as above, we get $\zeta_3$ is not a norm of an element of $k\setminus\lbrace 0\rbrace$ and $q^{\ast}=0$. Thus, $\operatorname{rank}\,(C_{k,3}^{(\sigma)})=1$.

 \subsubsection{Case where $d=q_1^{f_1}q_2^{f_2}q_3^{f_3} \equiv \pm1 \pmod  9$, with   $  q_1 \equiv q_2 \equiv 2\, \ \text{or} \ \,5 \pmod  9, q_3 \equiv -1  \pmod 9$ }
 As  $d \equiv \pm1 \pmod  9$, $\lambda$ is  not ramified in $k/k_0$.   In the same manner as above, we obtain $t=3$ and  $\zeta_3$ is not a norm of an element of $k\setminus\lbrace 0\rbrace$, therefore $q^{\ast}=0$. Thus $\operatorname{rank}\,(C_{k,3}^{(\sigma)})=1$.

\subsubsection{Case where $d=q_1^{f_1}q_2^{f_2}q_3^{f_3} \equiv \pm1 \pmod  9$, with $ q_1 \equiv q_2 \equiv q_3 \equiv 2 \, \text{or} \,  5  \pmod  9.$ }
 Reasoning as  in the above case, we get $t=3$ and $q^{\ast}=0$, then $\operatorname{rank}\,(C_{k,3}^{(\sigma)})=1$. This completes the proof of Theorem \ref{thm:Rank1}.



\section{The full $3$-class group of the Galois closure $k$ of $L$}
\label{s:FullClassGroup}

In Table \ref{tbl:Structures}
we finally summarize the structures of
the $3$-class group $C_{L,3}$ of the pure cubic field $L=\mathbb{Q}(\sqrt[3]{d})$
and the full $3$-class group $C_{k,3}$ of the normal closure $k=\mathbb{Q}(\sqrt[3]{d},\zeta_3)$,
which occur in dependence on the various shapes of the conductor $f$
in our Theorems \ref{thm:Ismaili1}--\ref{thm:Ismaili2},
where the $3$-class group $C_{k,3}^{(\sigma)}\simeq (3)$ of the ambiguous classes of $k/k_0$
is generally cyclic of order $3$.

In contrast to the previous tables,
it was not possible to compute the structure of the $3$-class group $C_{k,3}$
of the field $k$ of absolute degree $6$
by means of Voronoi's algorithm which is only able to determine
the fundamental unit $\varepsilon$ in the unit group $E_{L}=\langle -1,\varepsilon\rangle$,
the principal factorization type (briefly: type), the absolute principal factors,
and the class number $h_{L}$ of the field $L$ of degree $3$.
Therefore we had to employ the computational algebra system MAGMA \cite{BCP,BCFS,MAGMA}
for the numerical investigation of the normal closures $k$.

We emphasize that the paradigms for the radicand $d$ in Table \ref{tbl:Structures}
are ordered according to the values of the exponent $w$,
as opposed to the Tables \ref{tbl:Honda}, \ref{tbl:Ismaili1}, \ref{tbl:Ismaili2}.
For instance,
the smallest radicand $d=1605$ with $w=3$ occurs earlier than
the smallest radicand $d=2091$ with $w=2$,
for the conductor $f=9q_1q_2$ with $q_1\equiv 2,5\,(9)$ and $q_2\equiv 8\,(9)$.

The restriction of the exponent $w$ to the smallest value $w=1$
has been proved by Gerth \cite{Ge2005}
for conductors $f=3p_1$ with $p_1\equiv 4,7\,(9)$.
However, for conductors
$f=9p_1$ with $p_1\equiv 4,7\,(9)$ or
$f=p_1q_1$ with $p_1\equiv 4,7\,(9)$ and $q_1\equiv 2,5\,(9)$ or
$f=3q_1q_2$ with $q_1,q_2\equiv 2,5\,(9)$ or
$f=q_1q_2q_3$ with $q_1,q_2,q_3\equiv 2,5\,(9)$,
this restriction seems to be an \textit{open problem}
and can be stated as a \textit{conjecture} with
very strong computational support by our investigations.

\newpage


\renewcommand{\arraystretch}{1.0}

\begin{table}[ht]
\caption{Structure of $3$-class groups $C_{L,3}$ and $C_{k,3}$}
\label{tbl:Structures}

{\small

\begin{center}
\begin{tabular}{|cl|c|ccc|l|}
\hline
       & $f$                                                        & Type & $C_{L,3}$ & $C_{k,3}$ & $w$ & Paradigms for $d$         \\
\hline
\multicolumn{7}{|c|}{Items in Theorem \ref{thm:Ismaili1}} \\
\hline
 $(1)$ & $p_1\equiv 1\,(9)$                                         &          &         &                 &         &                             \\
       &                                                            & $\alpha$ & $(3^w)$ & $(3^w,3^{w-1})$ & $\ge 1$ & $19,199,3061,6733$          \\
       &                                                            & $\gamma$ & $(3^w)$ & $(3^w,3^{w})$   & $\ge 2$ & $541,8389$                  \\
\hline
 $(2)$ & $3p_1$, $p_1\equiv 4,7\,(9)$                               &          &         &                 &         &                             \\
       &                                                            & $\alpha$ & $(3)$   & $(3)$           &         & $7$                         \\
       &                                                            & $\beta$  & $(3)$   & $(3,3)$         &         & $61$                        \\
\hline
 $(3)$ & $9p_1$, $p_1\equiv 4,7\,(9)$                               &          &         &                 &         &                             \\
       &                                                            & $\alpha$ & $(3)$   & $(3)$           &         & $21$                        \\
       &                                                            & $\beta$  & $(3)$   & $(3,3)$         &         & $183$                       \\
\hline
 $(4)$ & $p_1q_1$, $p_1\equiv -q_1\equiv 4,7\,(9)$                  &          &         &                 &         &                             \\
       &                                                            & $\alpha$ & $(3)$   & $(3)$           &         & $26$                        \\
       &                                                            & $\beta$  & $(3)$   & $(3,3)$         &         & $62$                        \\
\hline
\multicolumn{7}{|c|}{Items in Theorem \ref{thm:Ismaili2}} \\
\hline
 $(1)$ & $9q_1$, $q_1\equiv 8\,(9)$                                 &          &         &                 &         &                             \\
       &                                                            & $\beta$  & $(3^w)$ & $(3^w,3^{w})$   & $\ge 1$ & $51,159,213$                \\
       &                                                            & $\gamma$ & $(3^w)$ & $(3^w,3^{w})$   & $\ge 1$ & $153,321,477$               \\
\hline
 $(2)$ & $q_1q_2$, $q_1,q_2\equiv 8\,(9)$                           &          &         &                 &         &                             \\
       &                                                            & $\beta$  & $(3^w)$ & $(3^w,3^{w})$   & $\ge 1$ & $901,8857,61273$            \\
       &                                                            & $\gamma$ & $(3^w)$ & $(3^w,3^{w})$   & $\ge 2$ & $1207,11917$                \\
\hline
 $(3)$ & $3q_1q_2$, $q_1,q_2\equiv 2,5\,(9)$                        & $\beta$  & $(3)$   & $(3,3)$         &         & $20$                        \\
 $(4)$ & $3q_1q_2$, $q_1\equiv 2,5$, $q_2\equiv 8$                  & $\beta$  & $(3^w)$ & $(3^w,3^{w})$   & $\ge 1$ & $34,535,1003,5972$          \\
\hline
 $(5)$ & $9q_1q_2$, $q_1\equiv 2,5\,(9)$                            & $\beta$  &         &                 &         &                             \\
       & $q_2\equiv 2,5\,(9)$                                       &          & $(3^w)$ & $(3^w,3^{w})$   & $\ge 1$ & $30,165,2514,7374$          \\
       & $q_2\equiv 8\,(9)$                                         &          & $(3^w)$ & $(3^w,3^{w})$   & $\ge 1$ & $102,2091,1605,17265,13833$ \\
\hline
 $(6)$ & $q_1q_2q_3$, $q_1,q_2,q_3\equiv 2,5\,(9)$                  & $\beta$  & $(3)$   & $(3,3)$         &         & $460$                       \\
 $(7)$ & $q_1q_2q_3$, $q_1,q_2\equiv 2,5$, $q_3\equiv 8$            & $\beta$  & $(3^w)$ & $(3^w,3^{w})$   & $\ge 1$ & $170,1394,4301,26452,46079$ \\

\hline
\end{tabular}
\end{center}

}

\end{table}
 


\section{Conclusion}
\label{s:Conclusion}

In this article, we have illuminated the genus theory and Galois cohomology
of an essential part of the pure cubic fields $\mathbb{Q}(\sqrt[3]{d})$,
namely those which possess $3$-class groups with abelian type invariants $1$, $(3)$, and $(3^w)$ with \(w\ge 2\).
As a result of our systematic analysis of ambiguous ideal class groups with $3$-rank one
of the normal closures $\mathbb{Q}(\sqrt[3]{d},\zeta_3)$ over $\mathbb{Q}(\zeta_3)$,
it turned out that these fields can be characterized by radicands $d$,
respectively conductors $f$, having at most three prime divisors,
and at most one among them congruent to $1$ modulo $3$.
Extensive computational investigations showed that
they constitute approximately one third (precisely $31.83\%=8.93\%+9.90\%+13.00\%$)
of the pure cubic fields $\mathbb{Q}(\sqrt[3]{d})$ with $d<10^6$.
(The asymptotic limit among all pure cubic fields may be different.)
An important \textit{open problem} is the proof of the conjectured equivalence for cubic residue symbols
in Corollary \ref{cor:Ismaili1}, mentioned in Example \ref{exm:Ismaili1}.
 


\section{Outlook}
\label{s:Outlook}

It should be pointed out that Ismaili \cite{Is} has investigated
another related scenario for the relative $3$-genus field $k^{\ast}$.
If the conductor $f$ is divisible by exactly one prime $p\equiv 1\,(\mathrm{mod}\,3)$, that is $s=1$,
but is not contained in Theorem \ref{thm:Ismaili1},
and if $C_{k,3}\simeq (3,3)$ is elementary bicyclic,
then the genus field coincides with the Hilbert $3$-class field, $k^{\ast}=k_1$,
and the composita $k\cdot L_1=k\cdot L_1^{\sigma}=k\cdot L_1^{\sigma^2}=K_4$
coincide with one of the four unramified cyclic cubic extensions $K_1,\ldots,K_4$ of $k$ within $k_1$,
as illustrated in Figure \ref{fig:IsmailiType3}
and studied in detail by Ismaili and El Mesaoudi \cite{IsEM}.
However, this situation lies outside of the present article,
because the $3$-group of ambiguous ideal classes has
$\operatorname{rank}\,(C_{k,3}^{(\sigma)})=2$
under the given conditions.
This will be the topic of our forthcoming paper \cite{AMI},
which also sheds light on peculiar phenomena
revealed by the lattice minima of the discrete geometric Minkowski image
of the maximal order \(\mathcal{O}_L\) of the pure cubic field \(L=\mathbb{Q}(\sqrt[3]{d})\).
 


\begin{figure}[ht]
\caption{Third possible location of the relative $3$-genus field $k^{\ast}=(k/k_0)^{\ast}$}
\label{fig:IsmailiType3}

{\tiny

\setlength{\unitlength}{1.0cm}
\begin{picture}(15,11)(-11,-10)

\put(-10,0.5){\makebox(0,0)[cb]{Degree}}

\put(-10,-2){\vector(0,1){2}}

\put(-10,-2){\line(0,-1){7}}
\multiput(-10.1,-2)(0,-1){8}{\line(1,0){0.2}}

\put(-10.2,-2){\makebox(0,0)[rc]{\(54\)}}
\put(-9.8,-2){\makebox(0,0)[lc]{Hilbert of bicubic}}
\put(-10.2,-4){\makebox(0,0)[rc]{\(18\)}}
\put(-9.8,-4){\makebox(0,0)[lc]{compositum}}
\put(-10.2,-5){\makebox(0,0)[rc]{\(9\)}}
\put(-9.8,-5){\makebox(0,0)[lc]{Hilbert of conjugate cubics}}
\put(-10.2,-6){\makebox(0,0)[rc]{\(6\)}}
\put(-9.8,-6){\makebox(0,0)[lc]{bicubic}}
\put(-10.2,-7){\makebox(0,0)[rc]{\(3\)}}
\put(-9.8,-7){\makebox(0,0)[lc]{conjugate cubics}}
\put(-10.2,-8){\makebox(0,0)[rc]{\(2\)}}
\put(-9.8,-8){\makebox(0,0)[lc]{quadratic}}
\put(-10.2,-9){\makebox(0,0)[rc]{\(1\)}}
\put(-9.8,-9){\makebox(0,0)[lc]{base}}

{\normalsize
\put(-3,0){\makebox(0,0)[cc]{Scenario III: \quad \(k^{\ast}=k_1\)}}
}



\put(0,-2){\circle*{0.2}}
\put(0.2,-2){\makebox(0,0)[lc]{\(k_1=k^{\ast}\)}}

\put(0,-2){\line(-3,-2){3}}
\put(0,-2){\line(-1,-2){1}}
\put(0,-2){\line(1,-2){1}}
\put(0,-2){\line(3,-2){3}}

\multiput(-3,-4)(2,0){4}{\circle*{0.2}}
\put(-3.2,-4){\makebox(0,0)[rc]{\(L_1\cdot k=L_1^{\sigma}\cdot k=L_1^{\sigma^2}\cdot k\)}}
\put(-1.2,-4){\makebox(0,0)[rc]{\(K_3\)}}
\put(1.2,-4){\makebox(0,0)[lc]{\(K_2\)}}
\put(3.2,-4){\makebox(0,0)[lc]{\(K_1\)}}

\put(0,-6){\line(-3,2){3}}
\put(0,-6){\line(-1,2){1}}
\put(0,-6){\line(1,2){1}}
\put(0,-6){\line(3,2){3}}

\put(0,-6){\circle*{0.2}}
\put(0.2,-6){\makebox(0,0)[lc]{\(k\)}}

\put(0,-6){\line(0,-1){2}}

\put(0,-8){\circle*{0.2}}
\put(0.2,-8){\makebox(0,0)[lc]{\(k_0\)}}

\put(-5,-5){\line(2,1){2}}
\put(-2,-7){\line(2,1){2}}
\put(-2,-9){\line(2,1){2}}


\put(-5,-5){\circle{0.2}}
\put(-5.2,-5){\makebox(0,0)[rc]{\(L_1\)}}
\put(-4.8,-5){\makebox(0,0)[lc]{\(L_1^{\sigma},L_1^{\sigma^2}\)}}

\put(-2,-7){\line(-3,2){3}}

\put(-2,-7){\circle{0.2}}
\put(-2.2,-7){\makebox(0,0)[rc]{\(L\)}}
\put(-1.8,-7){\makebox(0,0)[lc]{\(L^{\sigma},L^{\sigma^2}\)}}

\put(-2,-7){\line(0,-1){2}}

\put(-2,-9){\circle*{0.2}}
\put(-2.2,-9){\makebox(0,0)[rc]{\(\mathbb{Q}\)}}


\end{picture}

}

\end{figure}
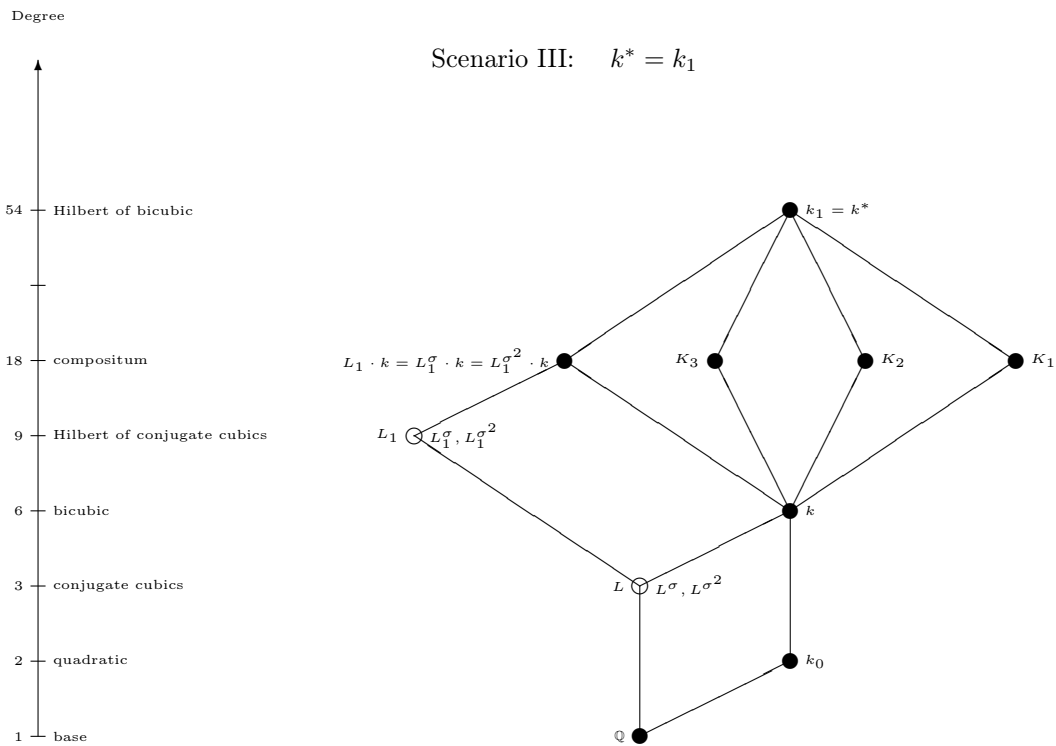



\section{Acknowledgements}
\label{s:Acknowledgements}

This paper is dedicated to the memory of Frank Emmett Gerth III, born 08 October 1945,
who contributed numerous great ideas to algebraic number theory, class field theory,
and in particular the theory of cubic field extensions,
starting with his Ph.D. thesis \cite{Ge1973} under supervision of Kenkichi Iwasawa,
and then continuously until \cite{Ge2005}, shortly before his early passing away on 23 May 2006.

The second author gratefully acknowledges that his research was supported
by the Austrian Science Fund (FWF): P 26008-N25.

\begin{quote}
Daniel C. MAYER \\
Naglergasse 53, 8010 Graz, Austria, \\
algebraic.number.theory@algebra.at \\
URL: http://www.algebra.at.\\

Abdelmalek AZIZI, Moulay Chrif ISMAILI and Siham AOUISSI  \\
Department of Mathematics and Computer Sciences, \\
Mohammed first University, \\
60000 Oujda - Morocco, \\
abdelmalekazizi@yahoo.fr, mcismaili@yahoo.fr, aouissi.siham@gmail.com. \\

Mohamed TALBI \\
Regional Center of Professions of Education and Training, \\
60000 Oujda - Morocco, \\
ksirat1971@gmail.com. 
\end{quote}

\end{document}